\documentclass[12pt]{article}
\usepackage{amsmath,amsthm,amscd,amssymb}
\usepackage[mathscr]{eucal}
\usepackage{amssymb}
\usepackage{latexsym}

\theoremstyle{definition}
\newtheorem{Def}{Definition}[section]
\newtheorem{Thm}[Def]{Theorem}
\newtheorem{Prop}[Def]{Proposition}
\newtheorem{Rem}[Def]{Remark}

\newtheorem{Cor}[Def]{Corollary}

\newtheorem{Lem}[Def]{Lemma}

\numberwithin{equation}{section}

\newcommand{\C}{\mathbb{C}}
\newcommand{\Z}{\mathbb{Z}}

\usepackage{longtable}

\begin{document}
\title{A ring of symmetric Hermitian modular forms of degree $2$ with integral Fourier coefficients}
\author{Toshiyuki Kikuta}
\maketitle


\begin{abstract}
We determine the structure over $\Z$ of the ring of symmetric Hermitian modular forms 
with respect to $\mathbb{Q}(\sqrt{-1})$ of degree $2$ (with a character),
whose Fourier coefficients are integers. 
Namely, we give a set of generators consisting of $24$ modular forms.
As an application of our structure theorem, 
we give the Sturm bounds of such the modular forms of weight $k$ with $4\mid k$, in the case $p=2$, $3$. 
We remark that the bounds for $p\ge 5$ are already known. 
\end{abstract}

\section{Introduction}
Let $e_4$, $e_6$ be the normalized Eisenstein series of respective weight $4$, $6$ for $\Gamma _1:=SL_2(\mathbb{Z})$ 
and $\delta $ the Ramanujan delta function defined by $\delta =2^{-6}\cdot 3^{-3}(e_4^3-e_6^2)$. 
For the $\Z$-module $M_k(\Gamma _1;\Z)$ consisting of modular forms of weight $k$ for $\Gamma _1$ whose Fourier coefficients are in $\Z$, 
we define an algebra over $\Z$ as
\[A(\Gamma _1;\Z):=\bigoplus_{k\in \Z}M_k(\Gamma _1;\Z). \] 
Then it is well-known as a classical result that all the Fourier coefficients of the modular forms $e_4$, $e_6$, $\delta $ are integers 
and they generate $A(\Gamma _1;\Z)$. 
Namely we have 
\[A(\Gamma _1;\Z)=\Z[e_4,e_6,\delta ]. \]

In the case of Siegel modular forms for the symplectic group $\Gamma _2:=Sp_2(\mathbb{Z})$ of degree $2$, there is a famous result of Igusa.  
He showed such the ring over $\Z$ are generated by $15$ modular forms. He showed also that its set of generators are minimal. 
 
In this paper, we consider the ring of symmetric Hermitian modular forms of degree $2$ with respect to $\mathbb{Q}(\sqrt{-1})$ 
whose Fourier coefficients are in $\mathbb{Z}$.
Since it seems to be difficult to give generators of the full space of them, 
we restrict our selves to the case where the weights are multiples of $4$.     
We remark that, the ring of Siegel modular forms whose weights are multiples of $4$ are generated over $\mathbb{Z}$ by 
$23$ modular forms. 
This is an easy conclusion of Igusa's result.   

In our case, there exists a set of generators consisting of $24$ modular forms whose weights are 
\begin{align*}
&4,\ 8,\ 12,\ 12,\ 12,\ 16,\ 16,\ 20,\ 24,\ 24,\ 28,\ 28,\ 32,\\
&36,\ 36,\ 36,\ 40,\ 40,\ 48,\ 48,\ 52,\ 60,\ 60,\ 72,\ 84. 
\end{align*}
The precise statement can be found in Theorem \ref{Thm1}. 
We construct explicitly these generators in Subsection \ref{Const}. 
  
As an application of this result, we can obtain the Sturm bounds for $p=2$, $3$, 
in the Hermitian modular forms whose weights are multiples of $4$ (Theorem \ref{Thm2}).
We remark that the Sturm bounds for $p\ge 5$ are already known in \cite{Ki-Na2}.  

\section{Preliminaries}
\label{sec:4}
\subsection{Hermitian modular forms of degree $2$}
\label{sec:4.1}
We deal with Hermitian modular forms of degree $2$ only for ${\boldsymbol K}:=\mathbb{Q}(\sqrt{-1})$. 
Let ${\mathcal O}$ be the ring of Gauss integers, i.e. ${\mathcal O}=\mathbb{Z}[\sqrt{-1}]$.

Let $\mathbb{H}_2$ be the Hermitian upper half-space of degree $2$ defined as
\[
\mathbb{H}_2:=\{ Z\in M_2(\mathbb{C})\;|\; \tfrac{1}{2i}(Z-{}^t\overline{Z})>0\;  \}
\]
where ${}^t\overline{Z}$ is the transposed complex conjugate of $Z$. 

The Hermitian modular group of degree $2$
\[
U_2(\mathcal{O}):=\left\{\;M\in M_{4}(\mathcal{O})\;|\; 
{}^t\overline{M}J_2M=J_2,\; J_2=\binom{\;0_2\;\,-1_2}{1_2\;\;\;0_2}\right\}
\]
acts on $\mathbb{H}_2$ by fractional transformation
\[
\mathbb{H}_2\ni Z\longmapsto M\langle Z\rangle :=(AZ+B)(CZ+D)^{-1},
\;M=\begin{pmatrix} A & B \\ C & D\end{pmatrix}\in U_2(\mathcal{O}).
\]
We denote by $M_k(U_2({\mathcal O}))=M_k(U_2({\mathcal O}),{\det}^k)$ the space of symmetric Hermitian modular forms of
weight $k$ and character ${\det}^k$ with respect to $U_2({\mathcal O})$.
(We deal with modular forms with character ${\det }^{k}$, but we drop this in the notation). 
Namely, it consists of holomorphic functions $F:\mathbb{H}_2\longrightarrow \mathbb{C}$ satisfying
\[
F\mid_kM(Z):={\det}(CZ+D)^{-k}F(M\langle Z\rangle )={\det (M)}^k\cdot F(Z),
\]
for all $M=\begin{pmatrix}A & B \\ C & D\end{pmatrix} \in U_2({\mathcal O})$ and $F({}^tZ)=F(Z)$.
Note that one has $M_k(U_2({\mathcal O}))=\{0\}$ if $k$ is odd. 
 
The cusp forms are characterized by the condition
\[
\Phi \Big(F\mid_k\binom{\!{}^t\overline{U}\;0}{0\;\;U}\Big) \equiv 0\quad \text{for}\;
\text{all}\;
U\in GL_2(\mathbb{Q}(\sqrt{-1}))
\]
where $\Phi$ is the Siegel $\Phi$-operator.
We denote by $S_k(U_2({\mathcal O}))$ the subspace consisting of all cusp forms in $M_k(U_2({\mathcal O}))$. 
\subsection{Fourier expansion}
\label{sec:2.2}

Since all of $F$ in $M_k(U_2({\mathcal O}))$ satisfies the condition
\[
F(Z+B)=F(Z) 
\quad \text{for}\;\text{all}\;
B\in Her_2(\mathcal{O}),
\]
it has a Fourier expansion of the form
\[
F(Z)=\sum_{0\leq H\in\Lambda_2(\boldsymbol{K})}a_F(H)e^{2\pi i\text{tr}(HZ)}, 
\]
where
\[
\Lambda_2(\boldsymbol{K}):=\{ H=(h_{ij})\in Her_2(\boldsymbol{K})\; |\; h_{ii}\in\mathbb{Z},
2 h_{ij}\in\mathcal{O} \}.
\]
We write $H=(m,r,s,n)$ for $H=\begin{pmatrix} m & \frac{r+si}{2} \\ \frac{r-si}{2} & n \end{pmatrix}\in \Lambda _2({\mathcal O})$ and also $a_F(m,r,s,n)$ for $a_F\begin{pmatrix}m & \frac{r+si}{2} \\ \frac{r-si}{2} &  n \end{pmatrix}$ simply. 
Let $R$ be a subring of $\C$, we define $M_k(U_2({\mathcal O});R)$ as an $R$-module of all of $F\in M_k(U_2({\mathcal O}))$ 
such that $a_F(H)\in R$ for any $H\in \Lambda _2({\mathcal O})$. 
We put also $S_k(U_2({\mathcal O});R):=M_k(U_2({\mathcal O});R)\cap S_k(U_2({\mathcal O}))$.

We put
\begin{align*}
&\dot{q}_{11}:=\exp(2\pi i z_{11}),\quad \dot{q}_{22}:=\exp(2\pi i z_{22}), \\
&\dot{q}_{12}:=\exp\left( 2\pi i \frac{z_{12}-z_{21}}{-2i} \right),\quad \ddot{q}_{12}:=\exp\left(2\pi i \frac{z_{12}+z_{21}}{2}\right).
\end{align*}
Then for $H=(m,r,s,n)$ we have 
\[
e^{2\pi i\text{tr}(HZ)}=\dot{q_{1}}^{m}\dot{q_{12}}^{r}\ddot{q_{12}}^{s}\dot{q_{2}}^{n}.
\]

Then any element $F\in M_k(U_2({\mathcal O});R)$ can be regarded as an element of
\[
R[\![\dot{\boldsymbol{q}}]\!]:=
R[\dot{q}_{12}^{\pm 1},\ddot{q}_{12}^{\pm }][\![\dot{q_{1}},\dot{q_{2}}]\!].
\]
This notation is useful to calculate the Fourier expansion of Hermitian modular forms. 

We consider the Hermitian Eisenstein series of degree $2$
\[
E_k(Z):=\sum_{M=\left(\begin{smallmatrix} * & * \\ C & D \end{smallmatrix}\right)} ({\det}M)^{\frac{k}{2}}{\det}(CZ+D)^{-k},\quad 
Z\in\mathbb{H}_2,
\]
where $k>4$ is even and $M=\begin{pmatrix} * & * \\ C & D \end{pmatrix}$ runs over a set of representatives of
$\left\{\begin{pmatrix} * & * \\ 0_2& * \end{pmatrix} \right\} \backslash U_2(\mathcal{O})$. 
Then we have
\[
E_k\in M_k(U_2(\mathcal{O})).
\]
Moreover $E_4\in M_4(U_2(\mathcal{O}))$ is constructed 
by the Maass lift (\cite{Kri}).
The Fourier coefficient of $E_k$ is given by the following formula:
\begin{Thm}[Krieg \cite{Kri}, Dern \cite{Der}]
\label{GHE} 
The Fourier coefficient $a_{E_k}(H)$ of $E_k$ is given as follows.
\begin{align*}
& a_{E_k}(H)\\
&=\begin{cases}
1 & \text{if}\;\; H=0_2,\\
\displaystyle 
-\frac{2k}{B_k}\,\sigma_{k-1}(\varepsilon (H))  & \text{if}\;\;
{\rm rank}(H)=1,\\
\displaystyle
\frac{4k(k-1)}{B_k\cdot B_{k-1,\chi_{-4}}}\sum_{0< d|\varepsilon (H)}
d^{k-1}
G_{\boldsymbol{K}}(k-2,4\,{\det}(H)/d^2) & \text{if}\;\; 
{\rm rank}(H)=2.
\end{cases}
\end{align*}
where\\
\quad$B_m$ is the $m$-th Bernoulli number,\\
\quad$B_{m,\chi_{-4}}$ is the $m$-th generalized Bernoulli number associated 
with the
Kronecker character $\chi_{-4}=\left(\frac{-4}{*}\right)$,\\
\quad $\varepsilon (H):={\rm max}\{ l\in\mathbb{N}\,|\, l^{-1}H\in
\Lambda_2(\boldsymbol{K})\,\}$,\\
and
\begin{equation}
\label{GK}
\begin{split}
&G_{\boldsymbol{K}}(m,N):=\frac{1}{1+|\chi_{-4}(N)|}
(\sigma_{m,\chi_{-4}}(N)-
\sigma^*_{m,\chi_{-4}}(N))\\
&\sigma_{m,\chi_{-4}}(N):=\sum_{0< d|N}\chi_{-4}(d)d^m,\quad
\sigma^*_{m,\chi_{-4}}(N):=\sum_{0< d|N}\chi_{-4}(N/d)d^m.
\end{split}
\end{equation}
\end{Thm}

We can construct cusp forms by the Hermitian Eisenstein series
(cf. \cite{D-K}, Corollary 2);
\begin{align*}
&E_{10}-E_4E_6\in S_{10}(U_2(\mathcal{O})),\\
&E_{12}-\frac{441}{691}E_4^3-\frac{250}{691}E_6^2\in S_{12}(U_2(\mathcal{O})).
\end{align*}

\subsection{Siegel modular forms of degree $2$}
Let $M_k(\Gamma_2)$ denote the space of Siegel modular forms of weight $k$
$(\in\mathbb{Z})$ for the Siegel modular group $\Gamma_2:=Sp_2(\mathbb{Z})$
and $S_k(\Gamma_2)$ the subspace of cusp forms.

Any Siegel modular form
$F$ in $M_k(\Gamma_2)$  has a Fourier expansion of the form
\[
F(Z)=\sum_{0\leq T\in\Lambda_2}a_F(T)e^{2\pi i\text{tr}(TZ)},
\]
where
\[
\Lambda_2=Sym_2^*(\mathbb{Z})
:=\{ T=(t_{ij})\in Sym_2(\mathbb{Q})\;|\; t_{ii},\;2t_{ij}\in\mathbb{Z}\; \}
\]
(the lattice in $Sym_2(\mathbb{R})$ of half-integral, symmetric matrices).
We write $T=(m,r,n)$ for $T=\begin{pmatrix}m & \frac{r}{2} \\ \frac{r}{2} & n \end{pmatrix}$ and also $a_F(m,r,n)$ for $a_F\begin{pmatrix}m & \frac{r}{2} \\ \frac{r}{2} & n\end{pmatrix}$.

Taking $q_{ij}:=\text{exp}(2\pi iz_{ij})$ with $Z=(z_{ij})\in\mathbb{H}_2$, we have for $T=(m,r,n)$
\[
e^{2\pi i\text{tr}(TZ)}=q_{11}^{m}q_{12}^{r}q_{22}^{n}.
\]
For any subring $R\subset\mathbb{C}$, we adopt the notation,
\begin{align*}
& M_k(\Gamma_2;R):=\{ F=\sum_{T\in\Lambda_n}a_F(T)q^T\;|\;
a_F(T)\in R\;(\forall T\in\Lambda_2)\;\},\\
& S_k(\Gamma_n;R):=M_k(\Gamma_2)\cap S_k(\Gamma_2).
\end{align*}

Any element $F\in M_k(\Gamma_2;R)$ can be regarded as an element of
\[
R[\![\boldsymbol{q}]\!]:=R[q_{12}^{-1},q_{12}][\![ q_{11},q_{22}]\!].
\]

The space
$\mathbb{H}_2$ contains the Siegel upper half-space of degree $2$
\[
\mathbb{S}_2:=\mathbb{H}_2\cap Sym_2(\mathbb{C}).
\]
Hence we can define the restriction map
\begin{align*}
R[\![\dot{\boldsymbol{q}}]\!]\longrightarrow R[\![\boldsymbol{q}]\!]
\end{align*}
via the correspondence $F\mapsto F|_{\mathbb{S}_2}:=F(z_{ij})|_{z_{21}=z_{12}}$ (this means $\dot{q}_{12}\mapsto 1$, $\ddot{q}_{12}\mapsto q_{12}$). 
In particular, if $F\in M_k(U_2({\mathcal O});R)\subset R[\![\dot{\boldsymbol{q}}]\!]$ then we have $F|_{\mathbb{S}_2}\in M_k(\Gamma _2;R)\subset R[\![\boldsymbol{q}]\!]$. 
This fact comes from each condition of the modularity.

\subsection{Igusa's generators over $\mathbb{Z}$}
\label{sec:3.1}
Let $k$ be an even integer with $k\ge 4$. The Siegel Eisenstein series 
\[
G_k(Z):=\sum_{M=\left(\begin{smallmatrix}*&* \\ C & D\end{smallmatrix}\right)} {\det}(CZ+D)^{-k},\quad Z\in\mathbb{S}_2
\]
defies an element of $M_k(\Gamma _2;\mathbb{Q})$.
Here $M=\begin{pmatrix}* & * \\ C & D\end{pmatrix}$ runs over a set of representatives
$\left\{\begin{pmatrix}* & * \\ 0_2 & * \end{pmatrix}\right\}\backslash\Gamma _2$. 
We write $X_4:=G_4$ and $X_6:=G_6$. 
We set
\begin{equation}
\label{Siegel cusp}
\begin{split}
X_{10}:&=-\frac{43867}{2^{10}\cdot 3^5\cdot 5^2\cdot 7\cdot 53}(G_{10}-G_4G_6), \\
X_{12}:&=-\frac{691\cdot 1847}{2^{13}\cdot 3^6\cdot 5^3\cdot 7^2}
(G_{12}-\frac{441}{691}G_4^3-\frac{250}{691}G_6^2).
\end{split}
\end{equation}
Then we have $X_k\in S_k(\Gamma_2;\Z)$ $(k=10,12)$ and
$a_{X_{10}}(1,1,1)=a_{X_{12}}(1,1,1)=1$.

Let $k$ be an even integer with $k\ge 4$ 
and $G_k$ the normalized Siegel Eisenstein series of weight $k$. 
We set  
\begin{align*}
&Y_{12} := 2^{-6}\cdot 3^{-3}(X_4^3 - X_6^2)+2^4\cdot 3^2X_{12},\\
&X_{16} := 2^{-2}\cdot 3^{-1}(X_4X_{12} - X_6 X_{10}),\\
&X_{18} := 2^{-2}\cdot 3^{-1}(X_6 X_{12}-X_4^2X_{10}),\\
&X_{24} := 2^{-3}\cdot 3^{-1}(X_{12}^2 - X_4 X_{10}^2),\\ 
&X_{28} := 2^{-1}\cdot 3^{-1}(X_4 X_{24} - X_{10} X_{18}),\\
&X_{30} := 2^{-1}\cdot 3^{-1}(X_6 X_{24} - X_4 X_{10} X_{16}),\\
&X_{36} :=2^{-1}\cdot 3^{-2}(X_{12} X_{24} - X_{10}^2 X_{16}),\\
&X_{40} :=2^{-2}(X_4 X_{36} - X_{10} X_{30}),\\
&X_{42} := 2^{-2}\cdot 3^{-1}(X_{12} X_{30} - X_4 X_{10} X_{28}),\\
&X_{48} := 2^{-2}(X_{12}X_{36} - X_{24}^2).
\end{align*}

We write 
\begin{align*}
A^{(m)}(\Gamma _2;\mathbb{Z}):=\bigoplus _{k\in m\mathbb{Z}}M_k(\Gamma _2;\Z).
\end{align*}
The following structure theorem is due to Igusa. 
\begin{Thm}[Igusa \cite{Igu2}] 
One has $X_k\in M_k(\Gamma _2;\Z)$ ($k=4$, $6$, $\cdots $, $48$) and $Y_{12}\in M_{12}(\Gamma _2;\Z)$ and the graded ring $A^{(2)}(\Gamma _2;\mathbb{Z})$
is generated over $\mathbb{Z}$ by them. Moreover, the set of $14$ generators are minimal. 
\end{Thm}
\begin{Rem}
Actually, he determined the structure of the full space $A^{(1)}(\Gamma _2;\Z)$ by using the cusp form of weight $35$. 
However, since we do not use this result we do not mention its detail.  
\end{Rem}
From his result, we have immediately the following property.
\begin{Cor}
\label{Cor:S_gen}
The ring $A^{(4)}(\Gamma _2;\mathbb{Z})$ is generated over $\mathbb{Z}$ by the following $23$ generators;  
\begin{align*}
&S_4:=X_4,\quad S_{12}:=X_{12},\quad T_{12}:=Y_{12},\quad U_{12}:=X_6^2,\quad 
S_{16}:=X_{10}X_6,\\
&T_{16}:=X_{16},\quad S_{20}:=X_{10}^2,\quad S_{24}:=X_{24},\quad T_{24}:=X_6X_{18},\\ 
&S_{28}:=X_{28},\quad T_{28}:=X_{10}X_{18},\quad S_{36}:=X_{36},\quad T_{36}:=X_{18}^2,\\
&U_{36}:=X_6X_{30},\quad S_{40}:=X_{40},\quad T_{40}:=X_{10}X_{30},\quad S_{48}:=X_{48},\\ 
&T_{48}:=X_{18}X_{30},\quad S_{52}:=X_{42}X_{10},\quad S_{60}:=X_{30}^2,\quad T_{60}:=X_{18}X_{42},\\
&S_{72}:=X_{30}X_{42},\quad S_{84}:=X_{42}^2.
\end{align*}
\end{Cor}

For later use, we introduce the Sturm bounds for Siegel modular forms of degree $2$. 
\begin{Thm}[Choi-Choie-Kikuta \cite{C-C-K}, Kikuta-Takemori \cite{Ki-Ta}]
\label{Stbd0}
Let $k$ be a positive integer and $p$ an any prime. Let $F\in M_{k}(\Gamma _2;\mathbb{Z}_{(p)})$. 
Suppose that $a_F(m,r,n)\equiv 0$ mod $p$ for any $m$, $r$, $n$  with 
\[m,\ n\le [k/10]\]
and $4mn-r^2\ge 0$. Then we have $F\equiv 0$ mod $p$.   
\end{Thm}

\subsection{Structure over $\Z[1/2,1/3]$}
We set $H_{4}:=E_4$ and 
\begin{align*}
&H_8 :=-\frac{61}{2^{10}\cdot 3^2 \cdot 5^2} (E_{8}-H_4^2),\\
&F_{10} := -\frac{277}{2^9\cdot 3^3\cdot 5^2 \cdot 7} (E_{10}-H_4 \cdot E_6),\\
&H_{12} := -\frac{19\cdot 691\cdot 2659}{2^{11}\cdot 3^7\cdot 5^3\cdot 7^2\cdot 73},\\
&~~~~~~~~~\times \left(E_{12}- 
   \frac{3^2\cdot 7^2}{691}H _4^3 -\frac{2\cdot 5^3}{691}H    _6^2
+ \frac{2^9\cdot 3^4\cdot 5^2\cdot 7^2\cdot 6791}{19\cdot 691\cdot 2659}H    _4\cdot H   _8\right).
\end{align*}
We define the graded ring $A^{(m)}(U_2({\mathcal O});R)$ over $R$ by  
\[A^{(m)}(U_2({\mathcal O});R)=\bigoplus _{k\in m\mathbb{Z}}M_k(U_2({\mathcal O});R).\]

\begin{Thm}[Kikuta-Nagaoka \cite{Ki-Na} (cf. Dern-Krieg \cite{D-K})] 
\label{Thm:Ki-Na}
Then all of $H_4$, $E_6$, $H_8$, $F_{10}$, $H_{12}$ have Fourier coefficients in $\mathbb{Z}$ and they generate the graded ring 
\[
A^{(2)}(U_2({\mathcal O});\Z [1/2,1/3]). \]
Moreover, these $5$ generators are algebraically independent over $\mathbb{C}$
and we have 
\[
H_4|_{\mathbb{S}_2}=X_4, \quad E_6|_{\mathbb{S}_2}=X_6,\quad 
H_8|_{\mathbb{S}_2}=0,\quad F_{10}|_{\mathbb{S}_2}=6X_{10},\quad H_{12}|_{\mathbb{S}_2}=X_{12}.
\]
\end{Thm}
\begin{Rem}
The ring $A^{(2)}(U_2({\mathcal O});R)$ coincides with the ring of the full space $A^{(1)}(U_2({\mathcal O});R)$ of symmetric Hermitian modular forms, because of $M_k(U_2({\mathcal O}))=\{0\}$ for odd $k$.  
\end{Rem}

Let $p$ be a prime and $\mathbb{Z}_{(p)}$ the localization of $\Z$ at the prime ideal $(p)=p\Z$, namely,
$\mathbb{Z}_{(p)}:=\mathbb{Q}\cap\mathbb{Z}_p$. 
The following lemma will be needed in later sections.
For a formal Fourier series of the form $F=\sum a_F(H)e^{2\pi i {\rm tr}(HZ)}$, 
we define $v_p(F)\in\mathbb{Z}$ as usual by
\begin{equation}
\label{vp}
v_p(F):=\underset{{H\in\Lambda_2({\mathcal O})}}{\text{inf}}\text{ord}_p(a_F(H)).
\end{equation}
\begin{Lem}
\label{Lem:ord}
For any $F_i=\sum a_
{F_i}(H)e^{2\pi i {\rm tr}(HZ)}$ ($i=1$, $2$) with $v_p(F_i)<\infty $, 
 we have
\[
v_p(F_1F_2)=v_p(F_1)+v_p(F_2).
\]
\end{Lem}
\begin{proof}
We can easily prove this property, 
if we define an order for two elements of $\Lambda _2({\mathcal O})$ in the same way as in \cite{Ki-Na2}.    
\end{proof}

We will need the Sturm bounds in the later sections.
\begin{Thm}[Kikuta-Nagaoka \cite{Ki-Na2}]
\label{Thm:Na-Ta}
Let $p$ be a prime with $p\ge 5$. 
Suppose that $F\in M_k(U_2({\mathcal O});\mathbb{Z}_{(p)})$ satisfies
that $a_F(m,r,s,n)\equiv 0$ mod $p$ for all $m$, $n\le [k/8]$. 
Then we have $F\equiv 0$ mod $p$.   
\end{Thm}
\begin{Rem}
In \cite{Ki-Na2} Theorem 2, 
we obtained the similar type bounds as this statement, but they are not same.
We can modify the proof in the similar way as in \cite{Na-Ta} Proposition 4.5. 
\end{Rem}
In general, the Sturm bounds imply the ordinary vanishing conditions.   
\begin{Cor}
\label{Cor:Na-Ta}
Suppose that $F\in M_k(U_2({\mathcal O});\mathbb{Q})$ satisfies
that $a_F(m,r,s,n)=0$ for all $m$, $n\le [k/8]$. 
Then we have $F=0$.   
\end{Cor}
\begin{proof}
We may apply Theorem \ref{Thm:Na-Ta} to $F$ for infinitely many primes $p\ge 5$. 
\end{proof}

\section{Structure over $\mathbb{Z}$}
\subsection{Construction of generators}
\label{Const}
We set 
\begin{align*}
&I_{12} := 2^{-6}\cdot 3^{-3}(H _4^3 - E_6^2) + 2^4\cdot 3^2\cdot H _{12},\\
&J_{12}:=E_6^2,\\
&H_{16}:=2^{-1}\cdot 3^{-1}(E_6F_{10}-H _4^2H _8)\\
&I_{16} := 2^{-2}\cdot 3^{-1}(H _4 H _{12} - H_{16}),\\
&H_{20}:=2^{-2} \cdot 3^{-2}(F_{10}^2-H _4 H _8^2 -2^{2}\cdot 3  H _8 H_{12}),\\
&H_{24}:=2^{-3}\cdot 3^{-1} (H_{12}^2- H_4 H_{20}) - 2^{-1}\cdot 3^{-1}H_8\cdot I_{16}.
\end{align*}
In order to construct further generators, we use  temporarily the alphabets $K$, $L$. 
\begin{align*}
&K_{14}:=2^{-1}\cdot 3^{-1}(H _4 F_{10}-E_6 H_8),\\ 
&K_{18}:=2^{-2}\cdot 3^{-1}(E_6 H _{12} - H _4 K_{14}),\\
&K_{22}:=2^{-1}\cdot 3^{-1}(F_{10} H_{12}-H_8 K_{14}),\\
&K_{26}:=2^{-1}\cdot 3^{-1}(F_{10} I_{16}-H_{8} K_{18}),\\
&K_{30}:=2^{-1}\cdot 3^{-1}(E_6H_{24}-K_{14}I_{16}) +3^{-1}H_8F_{10}I_{12},\\
&L_{30}:=2^{-1}\cdot 3^{-1}(F_{10} H_{20}-H_8 K_{22}),\\ 
&K_{34}:=2^{-1}\cdot 3^{-1}(F_{10}H_{24}-H_{8}K_{26}), \\
&K_{42}:=2^{-2}\cdot 3^{-1}(H_{12} K_{30}-K_{14} H_{28})- 
 2^{-1}H_8I_{12}K_{22}. 
\end{align*}
From these definition and Theorem \ref{Thm:Ki-Na}, it is easy to see that
\begin{align*}
&K_{14}|_{\mathbb{S}_2}=X_4X_{10},\quad K_{18}|_{\mathbb{S}_2}=X_{18},\quad K_{22}|_{\mathbb{S}_2}=X_{10}X_{12},\\
&K_{26}|_{\mathbb{S}_2}=X_6X_{16},\quad K_{30}|_{\mathbb{S}_2}=X_{30},\quad L_{30}|_{\mathbb{S}_2}=X_{10}^3,\\
&K_{34}|_{\mathbb{S}_2}=X_{10}X_{24},\quad K_{42}|_{\mathbb{S}_2}=X_{42}. 
\end{align*}
Finally we put 
\begin{align*}
&I_{24}:=E_6K_{18},\\
&H_{28}:=2^{-1}\cdot 3^{-1}(H_4H_{24} - I_{28}) - 3^{-1}H_8^2I_{12},\\
&I_{28}:=2^{-1}\cdot 3^{-1}(F_{10}\cdot K_{18}-H_4\cdot H_8\cdot I_{16}),\\
&H_{36}:=2^{-1}\cdot 3^{-2}(H_{12}H_{24} - H_{20}I_{16}) +
  7\cdot 3^{-2}H_{8}H_{28}+ 3^{-1}H_8^3H_{12},\\
&I_{36}:=K_{18}^2,\qquad J_{36}:=E_6K_{30}\\
&H_{40}:=2^{-2}(H_4H_{36} - \frac{1}{2\cdot 3}F_{10}\cdot K_{30}) - 
 5\cdot 2^{-3}\cdot 3^{-1}H_4\cdot H_8\cdot H_{28}, \\
&~~~~~~~~~~~~~~+ 2^{-2}\cdot {H_8}^3\cdot H_{16} + 2^{-1}H_8\cdot I_{12}\cdot H_{20},\\
&I_{40}:=2^{-1}\cdot 3^{-1}(F_{10}K_{30}-H_4\cdot H_8\cdot H_{28}),\\ 
&H_{48}:=2^{-2}(H_{12}\cdot H_{36}-H_{24}^2)-2^{-3}H_8(H_{12} H_{28}+ 2 H_{40} \\
&~~~~~~~~~~~~~~+4 H_{10}^2 H_{12} H_8- 2 H_{20} H_4 H_8^2 - 2 H_{12} H_4 H_8^3+ 4 H_{20} H_8 I_{12} \\
&~~~~~~~~~~~~~~~~~~~~~~~~~~~~+ 2 H_{12} H_8^2 I_{12} - H_{24} I_{16} - 2 H_8^3 I_{16} + 2 I_{40}),\\
&I_{48}:=K_{18}K_{30},\\
&H_{52}:=2^{-1}\cdot 3^{-1}(F_{10} K_{42} - 2 F_{10}^2 H_{12}^2 H_8 - 2^2 H_{12} H_{20} H_8 I_{12}\\
&~~~~~~~~~~~~~~- 5 H_{10} H_{22} H_8 I_{12} - H_{28} H_8 I_{16} - H_8^3 I_{12} I_{16}),\\
&H_{60}:=K_{30}^2,\qquad I_{60}:=K_{18}K_{42},\qquad H_{72}:=K_{30}K_{42},\qquad H_{84}:=K_{42}^2.
\end{align*}
By the definition of them and from Theorem \ref{Thm:Ki-Na}, 
we can easily confirm the following property. 
\begin{Prop} 
We have
\begin{align*}
H_{k_1}|_{\mathbb{S}_2}=S_{k_1},\ I_{k_2}|_{\mathbb{S}_2}=T_{k_2}\quad \text{and}\quad   J_{k_3}|_{\mathbb{S}_2}=U_{k_3} 
\end{align*}
for each $k_1$, $k_2$, $k_3$ with 
\begin{align*}
&k_1\in \{4,12,16,20,24,28,32,36,40,48,52,60,72,84\},\\
&k_2\in \{12,16,24,28,36,40,48,60\},\quad k_3\in \{12,36\}.
\end{align*}
\end{Prop}

\subsection{Integralities of generators}
\label{Int}
First our purpose is to prove that, all Fourier coefficients of the modular forms constructed in the previous subsection are integers. 
We start with proving several lemmas. 

We put $H_4=1+2^{4}\cdot 3S$, $E_6=1+2^3\cdot 3^2T$ with $S$, $T\in \Z[\![\dot{\boldsymbol{q}}]\!]$. 
\begin{Lem}
\label{Lem0}
We have $S\equiv T$ mod $2^2\cdot 3$. 
\end{Lem} 
\begin{proof}
For $H\in \Lambda _2({\mathcal O})$ with ${\rm rank}(H)=1$, we have 
\begin{align*}
&a_{H_4}(H)= 2^4\cdot 3 \cdot 5\sum _{0<d\mid \varepsilon (H)}d^{3}, \\
&a_{E_6}(H)=-2^3\cdot 3^2 \cdot 7 \sum _{0<d\mid \varepsilon (H)}d^{5}. 
\end{align*}
The assertion (for ${\rm rank}(H)=1$) follows from $5\equiv -7$ mod $2^2 \cdot 3$ and an application of the Euler congruence 
\[\sum _{0<d\mid \varepsilon (H)}d^{3}\equiv \sum _{0<d\mid \varepsilon (H)}d^{5} \bmod{2^2 \cdot 3}. \]

Let $H\in \Lambda _2({\mathcal O})$ with ${\rm rank}(H)=2$. Then 
\begin{align*}
&a_{H_4}(H)=-2^6\cdot 3\cdot 5\sum _{0<d\mid \varepsilon (H)}d^3G_{{\boldsymbol K}}(3,4\det H/d^2 ),\\
&a_{E_6}(H)=-2^5\cdot 3^2\cdot 5^{-1}\cdot 7 \sum _{0<d\mid \varepsilon (H)}d^5G_{{\boldsymbol K}}(5,4\det H/d^2 ).
\end{align*}
The Euler congruence implies that 
\[\sum _{0<d\mid \varepsilon (H)}d^3G_{{\boldsymbol K}}(3,4\det H/d^2 )\equiv \sum _{0<d\mid \varepsilon (H)}d^5G_{{\boldsymbol K}}(5,4\det H/d^2 ) \bmod 2^2\cdot 3. \] 
On the other hand, we have 
\begin{align*}
2^2\cdot 5\equiv 2^2\cdot 5^{-1}\cdot 7 \bmod{2^2\cdot 3}. 
\end{align*}
Therefore the assertion holds. 
\end{proof}

By this lemma, we can put $T=S+2^2\cdot 3 U$ with $U\in \Z[\![\dot{\boldsymbol{q}}]\!]$. 
Then we have 
\begin{align*}
&H_4=1+2^4\cdot 3 S, \\ 
&E_6=1+2^3\cdot 3^2 S+2^5\cdot 3^3 U. 
\end{align*}
This is the one of important fact for our arguments on integralities of generators. 
 
\paragraph{Forms of weight  $\bold{12}$}
We remark that $J_{12}=E_6^2\in M_{12}(U_2({\mathcal O});\Z)$ follows from $E_6\in M_{6}(U_2({\mathcal O});\Z)$. 
\begin{Lem}
\label{Lem1}
We have $I_{12}\in M_{12}(U_2({\mathcal O});\Z)$. 
\end{Lem}
\begin{proof} 
We know by Theorem \ref{Thm:Ki-Na} that $H_{12}\in M_{12}(U_2({\mathcal O});\Z)$. Hence, it suffices to prove that 
$2^{-6}\cdot 3^{-3}(H_4^3 -E_6^2)\in M_{12}(U_2({\mathcal O});\Z)$. 
This can be confirmed by the expansion  
\[2^{-6}\cdot 3^{-3}(H_4^3 -E_6^2)=S^2 + 64 S^3 - U - 72 S U - 432 U^2, \]
where $S$, $U$ is defined as above. 
\end{proof}

\paragraph{Forms of weight $\bold{14}$, $\bold{16}$, $\bold{18}$}
For the proof of their integralities, we use (as in \cite{Ki-Na}) the correspondence between the Maass space and the Kohnen plus subspace which given by Krieg \cite{Kri}. 
We review it briefly.

We define the congruence subgroup of $\Gamma _1=SL_2(\Z)$ with level $N$ ($N\in \mathbb{N}$) as 
\[\Gamma _0^{(1)}(N):=\left\{\begin{pmatrix} a&b\\ c& d\end{pmatrix}\in \Gamma _1 \left| \right. c\equiv 0 \bmod{N} \right\}. \]  
Let $M_k(\Gamma _0^{(1)}(4),\chi _{-4}^k)$ be the space of elliptic modular forms with character $\chi _{-4}^k$ for $\Gamma _0^{(1)}(4)$. 
Let ${\mathcal M}_k(U_2({\mathcal O}))$ be the Maass space consisting of all of $F \in M_k(U_2({\mathcal O}))$ satisfying the Maass relation. 
For the precise definition, see \cite{Kri}, p.676.

The Hermitian modular forms version of the Kohnen plus subspace is defined as
\begin{align*}
&M_k^+(\Gamma _0^{(1)}(4),\chi _{-4}^k)\\
&~~~~~~~:=\left\{f=\sum _{n=0}^{\infty} a_f(n)q^n \in M_k(\Gamma _0^{(1)}(4),\chi _{-4}^k) \left| \right. a_f(n)=0\ \forall n\equiv 1 \bmod{4} \right\}
\end{align*}
Krieg \cite{Kri} gave the isomorphism as vector spaces
\[M_{k-1}^+(\Gamma _0^{(1)}(4)\chi _{-4}^{k-1})\longrightarrow {\mathcal M}_k(U_2({\mathcal O})). \]

Let \[\theta :=1+2\sum _{n\ge 1}q^{n^2},\quad f_2:=\sum _{n\ge 1}\sigma _1(n)q^n \]
with $\sigma _1(n):=\sum _{0<d\mid n}d$ and $q:=e^{2\pi i \tau }$, $\tau \in  \mathbb{H} _1:=\{ \tau =x+iy \; | \; y>0 \}$. 
Then it is known that $\theta^2\in M_1(\Gamma _0^{(1)}(4),\chi _{-4})$ and $f_2\in M_2(\Gamma _0^{(1)}(4),1)$ and they generate the graded ring
\begin{align*}
\bigoplus _{k\in \mathbb{Z}}M_k(\Gamma _0^{(1)}(4),\chi _{-4}^k). 
\end{align*}

Hence we can construct a Hermitian modular form ${\rm Lift}(h) \in M_k(U_2({\mathcal O}))$ from a polynomial $h\in \C[\theta ^2, f_2]$ (such that $h\in M_{k-1}^+(\Gamma _0^{(1)}(4),\chi _{-4}^{k-1})$), 
by the relation between their Fourier coefficients
\[a_{{\rm Lift}(h)}(H)=\sum _{0<d \mid \varepsilon (H)}d^{k-1}\frac{1}{1+|\chi _{-4}(4\det H/d^2 )|}a_h(4\det H/d^2). \]

\begin{Lem}
\label{Lem2}
We have $I_{16} \in M_{16}(U_2({\mathcal O});\mathbb{Z})$ and $K_{k}\in M_{k}(U_2({\mathcal O});\mathbb{Z})$ for $k=14$, $18$. 
\end{Lem}
\begin{proof}
We set 
\begin{align*}
h_{15}&:=\theta^{14} f_2^4 -28\theta^{10} f_2^5 +192\theta^6f_2^6\\
 &=q^4 + 12 q^6 + 64 q^7 + 36 q^8 - 128 q^{10} - 1152 q^{11} - 
 936 q^{12} - 504 q^{14} \cdots\\
&=\sum _{n\ge 4}a_{h_{15}}(n)q^n.   
\end{align*}
Then we have $h_{15}\in M_{15}(\Gamma _0^{(1)}(4),\chi _{-4})$. 
By an easy numerical experiments, we can confirm that $a_{h_{15}}(n)=0$ for all $n$ with $n\le 500$ and $n\equiv 1$ mod $4$. 
In fact, we can prove $h_{15}\in M^+_{15}(\Gamma _0^{(1)}(4),\chi _{-4})$ as follows.
We consider
\begin{align*}
h_{15}+h_{15}|T_{\chi _{-4}}-h_{15}|U(2)V(2)=\sum _{n\equiv 1 \bmod{4}}a_{h_{15}(n)}q^n \in M_{15}(\Gamma _0^{(1)}(32),\chi _{-4}),
\end{align*}
where $U(l)$, $V(l)$ is the usual operators and $T_{\chi }$ is the twisting operator of the Dirichlet character $\chi $ given in Shimura \cite{Shim}. Namely, their action for 
$f=\sum _{n=0}^\infty a_f(n) q^n$
is described as 
\begin{align*}
&f|U(l)=\sum _{n=0}^\infty a_{f}(ln)q^n \\
&f|V(l)=\sum _{n=0}^\infty a_{f}(n)q^{ln}, \\
&f|T_{\chi }=\sum _{n=0}^\infty \chi (n)a_{f}(n)q^{n}.  
\end{align*} 
We remark that we have (at least) $f|U(l)\in M_k(\Gamma _0^{(1)}(Nl),\psi )$, $f|V(l)\in M_k(\Gamma _0^{(1)}(Nl^2),\psi )$ and $f|T_{\chi }\in M_k(\Gamma _0^{(1)}(Nl^2),\psi \chi ^2)$ when $f\in M_k(\Gamma ^{(1)}_0(N),\psi )$. 
 
Since the Sturm bound for $M_{15}(\Gamma _0^{(1)}(32),\chi _{-4})$ is 
\[\frac{15}{12}[\Gamma _1:\Gamma _0^{(1)}(32)]=\frac{15}{12}\cdot 32\left(1+\frac{1}{2}\right)=60,\]
our numerical experiment for $n\le 500$ is sufficient. 
Namely this shows that $\sum _{n\equiv 1 \bmod{4}}a_{h_{15}(n)}q^n=0$ and hence $h_{15}\in M_{15}^+(\Gamma _0^{(4)},\chi _{-4})$. 
Therefore we can apply the isomorphism constructed by Krieg,
there exists ${\rm Lift}(h_{15})\in M_{16}(U_2({\mathcal O}))$ satisfying that
\begin{align*}
a_{{\rm Lift}(h_{15})}(H)
=\sum _{0<d\mid \varepsilon (H)}\frac{d^{15}}{1+|\chi _{-4}(4\det H/d^2)|}
a_{h_{15}}(4\det H/d^2).
\end{align*} 

By the definition of $h_{15}$, we see that $h_{15}\equiv f_2^4$ mod $2$ because of $\theta \equiv 1$ mod $2$.
This implies immediately
\[\frac{1}{1+|\chi _{-4}(4\det H/d^2)|}a_{h_{15}}(4\det H/d^2)\in \mathbb{Z}\]
for each $d$. 
Namely ${\rm Lift}(h_{15})\in M_{16}(U_2({\mathcal O});\mathbb{Z})$ follows.
 
By a direct calculation, we see that
\[a_{I_{16}}(m,r,s,n)=a_{{\rm Lift}(h_{15})}(m,r,s,n)-56a_{H_8^2}(m,r,s,n)\]
for all $(m,r,s,n)\in \Lambda _{2}({\mathcal O})$ with $m$, $n\le 2=[16/8]$. 
Applying Corollary \ref{Cor:Na-Ta}, we obtain 
\[I_{16}={\rm Lift}(h_{15})-56H_{8}^2. \]
Since ${\rm Lift}(h_{15})-56H_{8}^2\in M_{16}(U_2({\mathcal O});\mathbb{Z})$, 
we have the assertion $I_{16}\in M_{16}(U_2({\mathcal O});\mathbb{Z})$.

Similarly, if we set  
\begin{align*}
&h_{13}:=2\theta^{14}f_2^3 - 60\theta^{10}f_2^4 + 448\theta^6f_2^5\in M^+_{13}(\Gamma _0^{(1)}(4),\chi _{-4})\\
&h_{17}:= \theta^{18} f_2^4 - 36\theta^{14}f_2^5 + 
     368\theta^{10}f_2^6 - 768\theta^6f_2^7\in M^+_{17}(\Gamma _0^{(1)}(4),\chi _{-4}), 
\end{align*}
then we can prove the following equalities 
\begin{align*}
&K_{14}={\rm Lift}(h_{13}), \\
&K_{18}={\rm Lift}(h_{17})+256H_8H_{10}. 
\end{align*}
The assertions for $K_{14}$, $K_{18}$ follow from these fact immediately. 

We will give numerical data we used in the proofs, in Subsection \ref{ProofPlus}
\end{proof}

\begin{Lem}
\label{Lem3}
We have \\
\quad 
(1) $H_{16}\in M_{16}(U_2({\mathcal O});\mathbb{Z})$, \\
\quad 
(2) $6 H_4 H_{12} - E_6 H_{10} + H_4^2 H_8\equiv 0$ mod $2^3\cdot 3^3$.  
\end{Lem}
\begin{proof}
(1) By the definition of $I_{16}$, we have 
\[2^2\cdot 3 I_{16}=H_4H_{12}-H_{16}. \]
Since $2^2\cdot 3I_{16}\equiv 0$ mod $2^2\cdot 3$ because of $I_{16}\in M_{16}(U_2({\mathcal O});\Z)$, we have $H_{16}\in M_{16}(U_2{(\mathcal O});\Z)$.

(2) By the definition of $H_{16}$, we have 
\[2\cdot 3 H_{16}=E_6F_{10}-H_4^2H_8. \]
Hence we can write as
 \[2^3\cdot 3^2 I_{16}=6H_4 H_{12}-E_6F_{10}+H_4^2H_8. \]
Since $I_{16}\in M_{16}(U_2({\mathcal O});\Z)$, we have $6H_4 H_{12}-E_6F_{10}+H_4^2H_8\equiv 0$ mod $2^3\cdot 3^2$.  
Using the fact that $H_4\equiv 1$ mod $2^4\cdot 3$, $E_6\equiv 1$  mod $2^3\cdot 3^2$, we get 
\[6H_{12}-F_{10}+H_4^2H_8 \equiv 0 \bmod{2^3\cdot 3^2}. \] 
\end{proof}

From (2) in this lemma, we may write as 
\[6H_{12}-F_{10}+H_4^2H_8=2^3\cdot 3^2 V\] 
with $V\in \Z[\![\dot{\boldsymbol{q}}]\!]$.  
This description is another important thing for our arguments. 
\paragraph{Forms of weight $\boldsymbol{k}$ with $\boldsymbol{k \ge 20}$}
First we remark that $I_{24}\in M_{24}(U_2({\mathcal O});\Z)$ is trivial because of $I_{24}=E_6K_{18}$ and $E_6\in M_{6}(U_2({\mathcal O});\Z)$, $K_{18}\in M_{18}(U_2({\mathcal O});\Z)$. Similarly, the integralities of $I_{36}=K_{18}^2$, $J_{36}=E_6K_{30}$, $I_{48}=K_{18}K_{30}$, $H_{60}=K_{30}^2$, $I_{60}=K_{18}K_{42}$, $H_{72}=K_{32}K_{42}$, $H_{84}=K_{42}^2$ follow from that of $E_6$, $K_{18}$, $K_{30}$, $K_{32}$, $K_{42}$. 
\begin{Lem}
We have the integralities of all the generators constructed in Section \ref{Subsec:gen}. 
\end{Lem}
\begin{proof}
By the definition of $H_{20}$, we can write as
\[H_{20}=2^{-2}\cdot 3^{-2}(F_{10}^2 - 12 H_{12} H_8 - H_4 H_8^2). \]
If we use the descriptions 
\begin{align*}
&F_{10}=6H_{12}+H_{4}^2H_8-2^3\cdot 3^2V,\\
&H_4=1+2^4\cdot 3 S, \\ 
&E_6=1+2^3\cdot 3^2 S+2^5\cdot 3^3 U, 
\end{align*}
then we have 
\begin{align*}
H_{20}&= H_{12}^2 + 32 H_{12} H_8 S + 4 H_8^2 S + 768 H_{12} H_8 S^2 + 384 H_8^2 S^2 \\
&+ 12288 H_8^2 S^3 + 147456 H_8^2 S^4 + 24 H_{12} V + 4 H_8 V + 
   384 H_8 S V \\
&+ 9216 H_8 S^2 V + 144 V^2. 
\end{align*}
This shows $H_{20}\in M_{20}(U_2({\mathcal O});\mathbb{Z})$. 

Similarly, we can prove the integralities of all the generators.
In fact we can confirm that, all the generators have descriptions as polynomials of $H_{12}$, $H_8$, $S$, $U$, $V\in \Z[\![\dot{\boldsymbol{q}}]\!]$ with integral coefficients (see Subsection \ref{List}).  
\end{proof}

Now we could prove the integralities of our generators: 
\begin{Thm}
All of the modular forms
\begin{align*}
&H_4,\ H_8,\ H_{12},\ I_{12},\ J_{12},\ H_{16},\ I_{16},\ H_{20},\ H_{24},\ I_{24},\ H_{28},\ I_{28},\\
&H_{36},\ I_{36},\ J_{36},\ H_{40},\ I_{40},\ H_{48},\ I_{48},\ H_{52},\ H_{60},\ I_{60},\ H_{72},\ H_{84}
\end{align*}
and also 
\begin{align*}
&K_{14},\ K_{18},\ K_{22},\ K_{26},\ K_{30},\ L_{30},\ K_{34},\ K_{42},\ K_{38}
\end{align*}
are elements of $\Z[\![\dot{\boldsymbol{q}}]\!]$. 
\end{Thm}
\subsection{Proof of the structure theorem}
We are now in a position to
 prove the following main result. 
\begin{Thm}
\label{Thm1}
The graded ring $A^{(4)}(U_2({\mathcal O});\mathbb{Z})$ over $\mathbb{Z}$ is generated by $24$ modular forms \begin{align*}
&H_4,\ H_8,\ H_{12},\ I_{12},\ J_{12},\ H_{16},\ I_{16},\ H_{20},\ H_{24},\ I_{24},\ H_{28},\ I_{28},\\
&H_{36},\ I_{36},\ J_{36},\ H_{40},\ I_{40},\ H_{48},\ I_{48},\ H_{52},\ H_{60},\ I_{60},\ H_{72},\ H_{84}. 
\end{align*}
In other words, for any $F\in M_k(U_2({\mathcal O});\mathbb{Z})$ with $4\mid k$, there exists a polynomial with $24$ variables having coefficients in $\mathbb{Z}$ such that
$F=P(H_4,H_8,H_{12},\cdots , H_{84})$.  
\end{Thm}
\begin{proof}
We prove it by an induction on the weight. 

For $k=4$, the statement is true clearly. 
Suppose that the statement is true for all $k$ with $k<k_0$. 
Let $F\in M_{k_0}(U_2({\mathcal O});\mathbb{Z})$. 
Then there exists a polynomial $P$ with $23$ variables having coefficients in $\mathbb{Z}$ such that $F|_{\mathbb{S}_2}=P(S_4,S_{12},T_{12},\cdots, S_{84})$ because of Corollary \ref{Cor:S_gen}.
Then we have $F-P(H_4,H_{12},I_{12},\cdots ,H_{84})\in M_{k_0}(U_2({\mathcal O};\mathbb{Z}))$ and $(F-P(H_4,H_{12},I_{12},\cdots , H_{84}))|_{\mathbb{S}_2}=0$. 
Therefore there exists $F'\in M_{k_0-8}(U_2({\mathcal O});\mathbb{Q})$ such that $F-P(H_4,H_{12},I_{12},\cdots , H_{84})=H_8F'$.
Since all Fourier coefficients of $P(H_4,H_{12},I_{12},\cdots , H_{84})$ are in $\mathbb{Z}$, we have $H_8F'\in M_{k}(U_2({\mathcal O});\mathbb{Z})$. 
By $v_p(H_8)=0$ for any prime $p$, we have $F'\in M_{k_0-8}(U_2({\mathcal O});\mathbb{Z})$ because of Lemma \ref{Lem:ord}.
By the induction hypothesis, there exists a polynomial $P'$ such that $F'=P'(H_4,H_8,H_{12},\cdots , H_{84})$. 
Therefore we have 
\[F=P(H_4,H_{12},I_{12},\cdots , H_{84})+H_8P'(H_4,H_8,H_{12}\cdots , H_{84}). \]
This completes the proof of Theorem \ref{Thm1}.   
 \end{proof}

\begin{Rem}
To determine the structure of $A^{(2)}(U_2({\mathcal O});\Z)$ by our method, we need $K_{46}\in M_{46}(U_2({\mathcal O});\Z)$ such that 
$K_{46}|_{\mathbb{S}_2}=X_{10}X_{36}$. 
However, we predict that there does not exist such $K_{46}$, due to the leading terms of Fourier expansions. 
This is a main reason why we restricted our selves to the case where the weights are multiples of $4$.
We remark also that we can construct $K_{46}'\in M_{46}(U_2({\mathcal O});\Z)$ such that $K_{46}'|_{\mathbb{S}_2}=3X_{10} X_{36}$. 
\end{Rem}

\subsection{An Application}
As an application,  we have the following Sturm bounds for any $k$ with $4\mid k$. 
\begin{Thm}
\label{Thm2}
Let $p$ be an any prime and $k$ an integer with $4\mid k$. 
Suppose that $F\in M_{k}(U_2({\mathcal O});\mathbb{Z})$ satisfies that $a_F(m,r,s,n)\equiv 0$ mod $p$
for all $m$, $n\in \mathbb{Z}$ with  
\[0\le m,\ n\le \left[\frac{k}{8}\right] \]
Then we have $F\equiv 0$ mod $p$. 
\end{Thm}
For the primes $p\ge 5$, we can prove the statement in the similar way. 
Hence we prove the essential case $p=2$, $3$ only.    

\begin{Lem}
\label{LemA1}
Let $p=2$, $3$ and $k$ be an even integer with $4\mid k$. 
Suppose that $F\in M_k(U_2({\mathcal O});\mathbb{Z})$ satisfies $F|_{\mathbb{S}_2}\equiv 0$ mod $p$, then there exists $F'\in M_{k-8}(U_2({\mathcal O});\mathbb{Z})$ such that $F\equiv H_8 F'$ mod $p$. 
\end{Lem}
\begin{proof}
For $k=4$, $8$, we have as free $\mathbb{Z}$-modules
\begin{align*}
&M_{4}(U_2({\mathcal O});\mathbb{Z})=H_4\mathbb{Z},\\
&M_{8}(U_2({\mathcal O});\mathbb{Z})=H_4^2\mathbb{Z}\oplus H_8\mathbb{Z}. 
\end{align*}
If $k\neq 8$ and $F\not \equiv 0$ mod $p$, then $F|_{\mathbb{S}_2}\equiv 0$ mod $p$ is impossible. 
If $k=8$, then $F|_{\mathbb{S}_2}\equiv 0$ mod $p$ is possible only if $F\equiv cH_8$ mod $p$ for some $c\in \mathbb{Z}$. 
Therefore the statements for $k=4$, $8$ are true.

We prove the case $k\ge 12$ with $4\mid k$. 
Since $F|_{\mathbb{S}_2}\equiv 0$ mod $p$, 
we have $\frac{1}{p} F|_{\mathbb{S}_2}\in M_k(\Gamma_2;\mathbb{Z})$.
By Corollary \ref{Cor:S_gen}, there exists an isobaric polynomial $P$ with coefficients in $\mathbb{Z}$ such that $\frac{1}{p}F|_{\mathbb{S}_2}=P(S_4,S_{12},\cdots ,S_{84})$. 
If we put 
\[G:=P(H_4,H_{12},\cdots ,H_{84}),\] 
then we have $G\in M_{k}(U_2({\mathcal O});\mathbb{Z})$ and $(F-pG)|_{\mathbb{S}_2}=0$. 
By the result of Dern-Krieg \cite{D-K}, there exists $F'\in M_{k-8}(U_2({\mathcal O});\mathbb{Q})$ such that $F-pG=H_8F'$.
Since $v_p(F-pG)\ge 0$ and $v_p(H_8)=0$ for any $p\ge 2$, 
it should be that $F'\in M_{k-8}(U_2({\mathcal O});\mathbb{Z})$. 
Then we have $F\equiv H_8F'$ mod $p$.

This competes the proof of Lemma \ref{LemA1}. 
\end{proof}

We prove Theorem \ref{Thm2}. 
\begin{proof}[Proof of Theorem \ref{Thm2}]
For $k=4$, $8$, we have as free $\mathbb{Z}$-modules
\begin{align*}
&M_{4}(U_2({\mathcal O});\mathbb{Z})=H_4\mathbb{Z},\\
&M_{8}(U_2({\mathcal O});\mathbb{Z})=H_4^2\mathbb{Z}\oplus H_8\mathbb{Z}.
\end{align*}
Since $H_4\equiv 1$ mod $p$ and 
$H_8\not \equiv c$ mod $p$ for any $c\in \Z$, the statements for $k=4$, $8$ are trivial.

Let $k\ge 12$.
From $[k/8]\ge [k/10]$, we can apply the Sturm bound in Theorem \ref{Stbd0} to $F|_{\mathbb{S}_2}$ and then we have $F|_{\mathbb{S}_2}\equiv 0$ mod $p$.
By Lemma \ref{LemA1}, there exists $F'\in M_{k-8}(U_2({\mathcal O});\mathbb{Z})$ such that $F\equiv H_8F'$ mod $p$.
Then $F'$ has the property that $a_{F'}(m,r,s,n)\equiv 0$ mod $p$ for any $m$, $n\in \mathbb{Z}$ with
\[0\le m,\ n\le \left[\frac{k}{8}\right]-1=\left[\frac{k-8}{8}\right].\]
This is due to the explicit form of the Fourier expansion of $H_8$ (the same reason as in \cite{Ki-Ta} Lemma 5.1); 
\begin{align*}
H_8&=\dot{q}_1 \dot{q}_2 (4 -2\dot{q}_{12}^{-1}- 2\dot{q}_{12} - 2\ddot{q}_{12}^{-1} \\
&+ \dot{q}_{12}^{-1}\ddot{q}_{12}^{-1} + \dot{q}_{12}\ddot{q}_{12}^{-1} - 2 \ddot{q}_{12} + \dot{q}_{12}^{-1}\ddot{q}_{12} + \dot{q}_{12}\ddot{q}_{12})+\cdots. 
\end{align*}
Note here that $4\mid k-8$ and we can apply the above argument to $F'$.
 
If we apply this argument repeatedly, we have $F\equiv 0$ mod $p$. 
This completes the proof of Theorem \ref{Thm2}.
\end{proof}

\section{Completion of the proofs by numerical data}
\subsection{Fourier expansions of $\boldsymbol{h_{13}}$, $\boldsymbol{h_{15}}$, $\boldsymbol{h_{17}}$}
\label{ProofPlus}
In the proof of Lemma \ref{Lem2}, we relied on the numerical data. 
Hence we give its data here.  

Let
\[b_k:=\frac{k}{12}[\Gamma _1:\Gamma _0^{(1)}(32)] \] 
be the Sturm bounds we mentioned in the proof of Lemma \ref{Lem2}. 
Then we have $b_{13}=52$, $b_{15}=60$, $b_{17}=68$.  
Therefore the following numerical data are sufficient for our purpose. 
{\footnotesize 
\begin{align*}
h_{13}&:=2\theta^{14}f_2^3 - 60\theta^{10}f_2^4 + 448\theta^6f_2^5\\&
=2q^3 - 4 q^4 + 112 q^6 - 4 q^7 - 432 q^8 - 640 q^{10} - 
 594 q^{11} + 5504 q^{12} - 4320 q^{14} + 9380 q^{15} - 20288 q^{16} \\&
 + 
 46848 q^{18} - 71622 q^{19} - 16200 q^{20} - 123376 q^{22} + 
 331668 q^{23} + 282112 q^{24} - 65664 q^{26} - 978492 q^{27} \\&
- 
 453376 q^{28} + 709600 q^{30} + 1749808 q^{31} - 1112832 q^{32} - 
 120064 q^{34} - 1329480 q^{35} + 3895356 q^{36} \\&
- 2315088 q^{38} 
- 1756316 q^{39} - 152160 q^{40} - 2846208 q^{42} + 7579934 q^{43} - 
 11366784 q^{44} + 16414816 q^{46} \\&
- 17552376 q^{47} + 10176512 q^{48} 
+  5875200 q^{50} + 33105284 q^{51} + 3775288 q^{52}+\cdots ,
\end{align*}
\begin{align*}
h_{15}&:=\theta^{14} f_2^4 -28\theta^{10} f_2^5 +192\theta^6f_2^6\\&
=q^4 + 12 q^6 + 64 q^7 + 36 q^8 - 128 q^{10} - 1152 q^{11} - 
 936 q^{12} - 504 q^{14} + 7872 q^{15} + 8144 q^{16} + 16128 q^{18}\\& - 
 18816 q^{19} - 32022 q^{20} - 121100 q^{22} - 51264 q^{23} + 
 26976 q^{24} + 464256 q^{26} + 408960 q^{27} + 258448 q^{28} \\&- 
 909576 q^{30} - 577024 q^{31} - 971712 q^{32} + 355072 q^{34} - 
 2085120 q^{35} + 525753 q^{36} + 2238876 q^{38}\\& 
+ 7869888 q^{39} + 4278504 q^{40} - 5027328 q^{42} - 853760 q^{43} - 9440856 q^{44} + 
 8767832 q^{46} - 36277632 q^{47}\\&
 - 1162368 q^{48} - 26012160 q^{50} + 
 46803840 q^{51} + 24912602 q^{52} + 40240728 q^{54} + 71676992 q^{55}\\&
 - 
 22735296 q^{56} + 47704960 q^{58} - 187329024 q^{59} + 8247408 q^{60}+\cdots , 
\end{align*}
\begin{align*}
h_{17}&:= \theta^{18} f_2^4 - 36\theta^{14}f_2^5 + 
     368\theta^{10}f_2^6 - 768\theta^6f_2^7\\
&=q^4 - 12 q^6 - 128 q^7 - 228 q^8 - 800 q^{10} - 768 q^{11} + 
 1872 q^{12} + 15576 q^{14} + 36480 q^{15} + 9296 q^{16} \\&
- 108864 q^{18} -  297216 q^{19} - 178110 q^{20} + 356140 q^{22} + 845952 q^{23} + 
 816576 q^{24} - 682656 q^{26}\\&
 + 1071360 q^{27} - 803744 q^{28} +  3381480 q^{30} - 12461056 q^{31} - 5338176 q^{32} - 23163968 q^{34} + 
 20912640 q^{35} \\&
+ 16663617 q^{36} + 79051812 q^{38} + 40330368 q^{39} 
+  2424120 q^{40} - 99195264 q^{42}\\&
 - 169433600 q^{43} - 64675536 q^{44} - 
 142870072 q^{46} + 63431424 q^{47} - 965376 q^{48} + 629961600 q^{50} \\&
+  381400320 q^{51} + 220457666 q^{52} - 671789592 q^{54} - 
 295596160 q^{55} + 283752576 q^{56} + 90976480 q^{58}\\&
 +  62678016 q^{59} - 1557183840 q^{60} - 135149088 q^{62} - 
 2319442560 q^{63} - 394334976 q^{64} - 99539136 q^{66} \\&
+  1338126080 q^{67} + 6624813570 q^{68}+\cdots. 
\end{align*}
}
\subsection{Proof of integralities of the generators}
\label{List}
In this subsection, we list the descriptions of our generators as polynomials with variables $H_{12}$, $H_8$, $S$, $U$, $V\in \Z[\![\dot{\boldsymbol{q}}]\!]$, 
where $S$, $U$, $V$ are defined by
\begin{align*}
&F_{10}=6H_{12}+H_{4}^2H_8-2^3\cdot 3^2V,\\
&H_4=1+2^4\cdot 3 S, \\ 
&E_6=1+2^3\cdot 3^2 S+2^5\cdot 3^3 U. 
\end{align*}
The list below shows that the integralities of corresponding generators as in Subsection \ref{Int}.  
Namely we prove that our generators are elements of the ring $\Z[H_{12},H_8,S,U,V]$ in the following. 

{\tiny
\begin{align*}
K_{22}&=H_{12}^2 + 8 H_{12} H_8 S - 2 H_8^2 S + 384 H_{12} H_8 S^2
- 192 H_8^2 S^2 - 
   3072 H_8^2 S^3 + 24 H_8^2 U + 12 H_{12} V - 2 H_8 V - 96 H_8 S V, \\
H_{24}&= -2 H_{12}^2 S + H_{12} H_8 S + 96 H_{12} H_8 S^2 + 8 H_8^2 S^2 + 1536 H_{12} H_8 S^3
+ 896 H_8^2 S^3 + 30720 H_8^2 S^4 + 294912 H_8^2 S^5 - 12 H_{12} H_8 U\\
& - 2 H_8^2 U - 192 H_8^2 S U - 4608 H_8^2 S^2 U + H_{12} V + 48 H_{12} S V+ 12 H_8 S V + 1152 H_8 S^2 V + 18432 H_8 S^3 V - 144 H_8 U V + 6 V^2\\&
 + 
 288 S V^2, \\
K_{26}&=2 H_{12}^2 S + H_{12} H_8 S + 96 H_{12} H_8 S^2 + 8 H_8^2 S^2 + 3072 H_{12} H_8 S^3+ 
 1280 H_8^2 S^3 + 61440 H_8^2 S^4 + 884736 H_8^2 S^5 + 72 H_{12}^2 U\\&
 + 36 H_{12} H_8 U + 4 H_8^2 U + 2304 H_{12} H_8 S U + 480 H_8^2 S U + 
 55296 H_{12} H_8 S^2 U + 27648 H_8^2 S^2 U + 884736 H_8^2 S^3 U + 10616832 H_8^2 S^4 U \\
&+ H_{12} V + 96 H_{12} S V + 24 H_8 S V + 
 2304 H_8 S^2 V+ 55296 H_8 S^3 V + 1728 H_{12} U V + 288 H_8 U V + 
 27648 H_8 S U V\\&
 + 663552 H_8 S^2 U V + 12 V^2 + 864 S V^2 + 10368 U V^2,\\
H_{28}&=-48 H_{12} H_8^2 + 16 H_{12}^2 S^2 + 8 H_{12} H_8 S^2 + H_8^2 S^2 + 
 640 H_{12} H_8 S^3+ 192 H_8^2 S^3 + 12288 H_{12} H_8 S^4 + 12288 H_8^2 S^4\\
&+ 
 294912 H_8^2 S^5 + 2359296 H_8^2 S^6- 12 H_{12}^2 U - 4 H_{12} H_8 U - 
 288 H_{12} H_8 S U - 24 H_8^2 S U - 4608 H_{12} H_8 S^2 U\\
&- 2304 H_8^2 S^2 U - 
 36864 H_8^2 S^3 U + 144 H_8^2 U^2 + 4 H_{12} S V + 2 H_8 S V 
+384 H_{12} S^2 V + 288 H_8 S^2 V + 12288 H_8 S^3 V + 147456 H_8 S^4 V \\
&- 144 H_{12} U V- 24 H_8 U V - 1152 H_8 S U V + V^2 + 96 S V^2 + 2304 S^2 V^2,\\
I_{28}&=-2 H_{12}^2 S - H_{12} H_8 S - 192 H_{12}^2 S^2 - 192 H_{12} H_8 S^2 - 
 16 H_8^2 S^2- 9984 H_{12} H_8 S^3 - 2560 H_8^2 S^3 - 147456 H_{12} H_8 S^4 - 
 147456 H_8^2 S^4 \\
&- 3538944 H_8^2 S^5 - 28311552 H_8^2 S^6 + 
 72 H_{12}^2 U + 36 H_{12} H_8 U + 4 H_8^2 U + 2304 H_{12} H_8 S U + 
 576 H_8^2 S U + 27648 H_{12} H_8 S^2 U \\&
+ 27648 H_8^2 S^2 U 
+ 442368 H_8^2 S^3 U - H_{12} V 
- 120 H_{12} S V - 24 H_8 S V - 
 4608 H_{12} S^2 V - 3456 H_8 S^2 V - 147456 H_8 S^3 V\\
&- 
 1769472 H_8 S^4 V + 864 H_{12} U V+ 288 H_8 U V + 13824 H_8 S U V - 
 12 V^2 - 1152 S V^2 - 27648 S^2 V^2,
\end{align*}
\begin{align*}
K_{30}&=288 H_{12}^2 H_8 + 48 H_{12} H_8^2 + 4608 H_{12} H_8^2 S - 8 H_{12}^2 S^2 + 
 2 H_{12} H_8 S^2+ H_8^2 S^2 + 110592 H_{12} H_8^2 S^2 + 256 H_{12} H_8 S^3 + 
 192 H_8^2 S^3 \\&
+ 6144 H_{12} H_8 S^4+ 13056 H_8^2 S^4 + 368640 H_8^2 S^5 + 
 3538944 H_8^2 S^6 + 12 H_{12}^2 U + 2 H_{12} H_8 U + 288 H_{12}^2 S U + 
 240 H_{12} H_8 S U\\&
+ 13824 H_{12} H_8 S^2 U 
+ 1152 H_8^2 S^2 U + 
 221184 H_{12} H_8 S^3 U+ 129024 H_8^2 S^3 U + 4423680 H_8^2 S^4 U + 
 42467328 H_8^2 S^5 U- 864 H_{12} H_8 U^2\\
& - 144 H_8^2 U^2- 13824 H_8^2 S U^2- 331776 H_8^2 S^2 U^2 + 3456 H_{12} H_8 V + 4 H_{12} S V + 
 2 H_8 S V + 192 H_{12} S^2 V+ 312 H_8 S^2 V\\&
 + 15360 H_8 S^3 V + 
 221184 H_8 S^4 V + 144 H_{12} U V+ 6912 H_{12} S U V+ 1728 H_8 S U V + 
 165888 H_8 S^2 U V + 2654208 H_8 S^3 U V\\&
 - 10368 H_8 U^2 V + V^2 + 
 120 S V^2+ 3456 S^2 V^2 + 864 U V^2 + 41472 S U V^2,\\
L_{30}&=H_{12}^3 + H_{12}^2 H_8 + 48 H_{12}^2 H_8 S + 16 H_{12} H_8^2 S - H_8^3 S + 
 1152 H_{12}^2 H_8 S^2 + 1344 H_{12} H_8^2 S^2 - 32 H_8^3 S^2 + 
 36864 H_{12} H_8^2 S^3+ 7168 H_8^3 S^3\\&
 + 442368 H_{12} H_8^2 S^4 + 
 368640 H_8^3 S^4 + 7077888 H_8^3 S^5 + 56623104 H_8^3 S^6 + 20 H_8^3 U + 
 36 H_{12}^2 V + 18 H_{12} H_8 V - H_8^2 V + 1152 H_{12} H_8 S V \\
&+ 96 H_8^2 S V + 
 27648 H_{12} H_8 S^2 V + 13824 H_8^2 S^2 V + 442368 H_8^2 S^3 V + 
 5308416 H_8^2 S^4 V + 432 H_{12} V^2 + 72 H_8 V^2 + 6912 H_8 S V^2 \\&
+ 
 165888 H_8 S^2 V^2 + 1728 V^3,\\
K_{34}&=-2 H_{12}^3 S - H_{12}^2 H_8 S - 128 H_{12}^2 H_8 S^2 - 24 H_{12} H_8^2 S^2 - 
 2304 H_{12}^2 H_8 S^3- 2560 H_{12} H_8^2 S^3 - 64 H_8^3 S^3 - 
 92160 H_{12} H_8^2 S^4\\&
 - 12288 H_8^3 S^4 
- 884736 H_{12} H_8^2 S^5 - 
 737280 H_8^3 S^5 - 16515072 H_8^3 S^6 - 113246208 H_8^3 S^7 + 
 24 H_{12}^2 H_8 U + 10 H_{12} H_8^2 U+ H_8^3 U\\&
 + 768 H_{12} H_8^2 S U 
+ 
 144 H_8^3 S U + 18432 H_{12} H_8^2 S^2 U + 9216 H_8^3 S^2 U 
+ 
 294912 H_8^3 S^3 U + 3538944 H_8^3 S^4 U - H_{12}^2 V - 72 H_{12}^2 S V\\&
 - 
 32 H_{12} H_8 S V + 2 H_8^2 S V - 3456 H_{12} H_8 S^2 V - 96 H_8^2 S^2 V - 
 55296 H_{12} H_8 S^3 V - 27648 H_8^2 S^3 V 
- 1105920 H_8^2 S^4 V - 
 10616832 H_8^2 S^5 V\\&
 + 576 H_{12} H_8 U V + 96 H_8^2 U V
 + 
 9216 H_8^2 S U V 
+ 221184 H_8^2 S^2 U V - 18 H_{12} V^2 + H_8 V^2 - 
 864 H_{12} S V^2 - 144 H_8 S V^2 
- 20736 H_8 S^2 V^2\\&
 - 
 331776 H_8 S^3 V^2 + 3456 H_8 U V^2\\&
 - 72 V^3 - 3456 S V^3,\\
H_{36}&=-37 H_{12} H_8^3 + 16 H_{12}^2 H_8 S^2 + 8 H_{12} H_8^2 S^2 + H_8^3 S^2 + 
 128 H_{12}^2 H_8 S^3+ 704 H_{12} H_8^2 S^3 + 192 H_8^3 S^3 + 
 17408 H_{12} H_8^2 S^4 + 12800 H_8^3 S^4\\& 
+ 98304 H_{12} H_8^2 S^5
 + 
 352256 H_8^3 S^5 + 4194304 H_8^3 S^6 + 18874368 H_8^3 S^7 + 4 H_{12}^3 U - 
 8 H_{12}^2 H_8 U 
- 3 H_{12} H_8^2 U + 192 H_{12}^2 H_8 S U \\&
- 176 H_{12} H_8^2 S U
- 16 H_8^3 S U+ 4608 H_{12}^2 H_8 S^2 U + 768 H_{12} H_8^2 S^2 U - 
 1280 H_8^3 S^2 U + 147456 H_{12} H_8^2 S^3 U + 10240 H_8^3 S^3 U \\&
+ 
 1769472 H_{12} H_8^2 S^4 U + 1474560 H_8^3 S^4 U + 28311552 H_8^3 S^5 U + 
 226492416 H_8^3 S^6 U + 112 H_8^3 U^2 + 4 H_{12}^2 S V + 6 H_{12} H_8 S V+ 
 2 H_8^2 S V\\&
 + 576 H_{12} H_8 S^2 V + 304 H_8^2 S^2 V + 6144 H_{12} H_8 S^3 V + 
 14848 H_8^2 S^3 V + 270336 H_8^2 S^4 V + 1769472 H_8^2 S^5 V + 
 144 H_{12}^2 U V - 72 H_{12} H_8 U V \\&
- 16 H_8^2 U V + 4608 H_{12} H_8 S U V- 
 192 H_8^2 S U V 
+ 110592 H_{12} H_8 S^2 U V + 55296 H_8^2 S^2 U V + 
 1769472 H_8^2 S^3 U V + 21233664 H_8^2 S^4 U V \\&
+ H_{12} V^2 + H_8 V^2 + 
 96 H_{12} S V^2 + 120 H_8 S V^2 + 4608 H_8 S^2 V^2 + 55296 H_8 S^3 V^2 + 
 1728 H_{12} U V^2 + 288 H_8 U V^2 + 27648 H_8 S U V^2\\&
 + 
 663552 H_8 S^2 U V^2 + 8 V^3 + 576 S V^3 + 6912 U V^3,\\
K_{38}&=-96 H_{12}^2 H_8^2 - 16 H_{12} H_8^3 - 1536 H_{12} H_8^3 S + 16 H_{12}^3 S^2 + 
 12 H_{12}^2 H_8 S^2 
+ 2 H_{12} H_8^2 S^2 - 36864 H_{12} H_8^3 S^2 + 
 896 H_{12}^2 H_8 S^3\\&
 + 384 H_{12} H_8^2 S^3+ 16 H_8^3 S^3 + 
 18432 H_{12}^2 H_8 S^4 + 26624 H_{12} H_8^2 S^4 + 3328 H_8^3 S^4 + 
 737280 H_{12} H_8^2 S^5 + 258048 H_8^3 S^5 \\&
+ 7077888 H_{12} H_8^2 S^6 + 
 9240576 H_8^3 S^6
+ 150994944 H_8^3 S^7 + 905969664 H_8^3 S^8 -  12 H_{12}^3 U 
- 8 H_{12}^2 H_8 U - H_{12} H_8^2 U - 528 H_{12}^2 H_8 S U\\&
 - 
 176 H_{12} H_8^2 S U - 4 H_8^3 S U
- 9216 H_{12}^2 H_8 S^2 U 
- 
 11520 H_{12} H_8^2 S^2 U - 960 H_8^3 S^2 U - 258048 H_{12} H_8^2 S^3 U - 
 73728 H_8^3 S^3 U \\&
- 1769472 H_{12} H_8^2 S^4 U - 2211840 H_8^3 S^4 U - 
 21233664 H_8^3 S^5 U
+ 288 H_{12} H_8^2 U^2 + 48 H_8^3 U^2 
+ 
 4608 H_8^3 S U^2 + 110592 H_8^3 S^2 U^2\\&
 - 1152 H_{12} H_8^2 V + 
 4 H_{12}^2 S V + 2 H_{12} H_8 S V + 576 H_{12}^2 S^2 V + 480 H_{12} H_8 S^2 V 
+ 
 40 H_8^2 S^2 V + 27648 H_{12} H_8 S^3 V + 7168 H_8^2 S^3 V \\&
+ 
 442368 H_{12} H_8 S^4 V + 442368 H_8^2 S^4 V + 10616832 H_8^2 S^5 V + 
 84934656 H_8^2 S^6 V - 288 H_{12}^2 U V \
- 120 H_{12} H_8 U V - 4 H_8^2 U V\\&
 - 
 8064 H_{12} H_8 S U V - 1152 H_8^2 S U V - 110592 H_{12} H_8 S^2 U V 
- 
 82944 H_8^2 S^2 U V - 1327104 H_8^2 S^3 U V + 3456 H_8^2 U^2 V \\&
+ 
 H_{12} V^2 + 144 H_{12} S V^2 + 36 H_8 S V^2 + 6912 H_{12} S^2 V^2 + 
 5184 H_8 S^2 V^2 + 221184 H_8 S^3 V^2 + 2654208 H_8 S^4 V^2 
- 
 1728 H_{12} U V^2 \\&
- 432 H_8 U V^2 - 20736 H_8 S U V^2 + 12 V^3 + 
 1152 S V^3 
+ 27648 S^2 V^3,\\
H_{40}&=-24 H_{12}^2 H_8^2 - H_{12} H_8^3 - 42 H_{12} H_8^3 S + 3 H_8^4 S + 2 H_{12}^3 S^2 + 
 H_{12}^2 H_8 S^2 + 288 H_8^4 S^2 + 64 H_{12}^2 H_8 S^3 + 8 H_{12} H_8^2 S^3 + 
 6912 H_8^4 S^3 \\
&+ 768 H_{12}^2 H_8 S^4 + 512 H_{12} H_8^2 S^4 - 64 H_8^3 S^4 - 
 6144 H_{12} H_8^2 S^5 - 10240 H_8^3 S^5- 294912 H_{12} H_8^2 S^6 - 
 573440 H_8^3 S^6 - 13369344 H_8^3 S^7\\&
 - 113246208 H_8^3 S^8 - 
 2 H_{12}^3 U 
- H_{12}^2 H_8 U + 216 H_{12} H_8^3 U + 36 H_8^4 U - 24 H_{12}^3 S U - 
 84 H_{12}^2 H_8 S U - 14 H_{12} H_8^2 S U - H_8^3 S U + 3456 H_8^4 S U \\&
- 
 2304 H_{12}^2 H_8 S^2 U - 1632 H_{12} H_8^2 S^2 U - 176 H_8^3 S^2 U + 
 82944 H_8^4 S^2 U - 27648 H_{12}^2 H_8 S^3 U - 55296 H_{12} H_8^2 S^3 U - 
 12032 H_8^3 S^3 U\\&
 - 1105920 H_{12} H_8^2 S^4 U - 466944 H_8^3 S^4 U - 
 10616832 H_{12} H_8^2 S^5 U - 12386304 H_8^3 S^5 U - 
 198180864 H_8^3 S^6 U - 1358954496 H_8^3 S^7 U\\
& + 72 H_{12} H_8^2 U^2 + 
 4 H_8^3 U^2 + 192 H_8^3 S U^2 + 3 H_8^3 V + 216 H_8^3 S V + 
 24 H_{12}^2 S^2 V - 2 H_8^2 S^2 V - 384 H_{12} H_8 S^3 V - 416 H_8^2 S^3 V \\
&- 
 18432 H_{12} H_8 S^4 V - 30720 H_8^2 S^4 V - 958464 H_8^2 S^5 V - 
 10616832 H_8^2 S^6 V - 36 H_{12}^2 U V - 12 H_{12} H_8 U V - H_8^2 U V + 
 2592 H_8^3 U V\\&
 - 864 H_{12}^2 S U V - 1152 H_{12} H_8 S U V - 
 168 H_8^2 S U V - 41472 H_{12} H_8 S^2 U V - 12672 H_8^2 S^2 U V - 
 663552 H_{12} H_8 S^3 U V - 552960 H_8^2 S^3 U V\\&
 - 
 13271040 H_8^2 S^4 U V - 127401984 H_8^2 S^5 U V - 6 H_{12} S V^2 - 
 3 H_8 S V^2 - 288 H_{12} S^2 V^2 - 432 H_8 S^2 V^2 - 20736 H_8 S^3 V^2 \\&
- 
 331776 H_8 S^4 V^2 - 216 H_{12} U V^2 - 36 H_8 U V^2 - 
 10368 H_{12} S U V^2 - 5184 H_8 S U V^2 - 248832 H_8 S^2 U V^2 - 
 3981312 H_8 S^3 U V^2 - V^3 \\
&- 120 S V^3 - 3456 S^2 V^3 - 864 U V^3 - 
 41472 S U V^3,
\end{align*}

\begin{align*}
I_{40}&=288 H_{12}^3 H_8 + 96 H_{12}^2 H_8^2 + 16 H_{12} H_8^3 + 9216 H_{12}^2 H_8^2 S + 
 1920 H_{12} H_8^3 S - 8 H_{12}^3 S^2 - 2 H_{12}^2 H_8 S^2+ 
 221184 H_{12}^2 H_8^2 S^2\\& 
+ 110592 H_{12} H_8^3 S^2 + 96 H_{12} H_8^2 S^3 + 
 8 H_8^3 S^3 + 3538944 H_{12} H_8^3 S^3 + 3072 H_{12}^2 H_8 S^4 
+ 
 11776 H_{12} H_8^2 S^4 + 2048 H_8^3 S^4 \\&
+ 42467328 H_{12} H_8^3 S^4 + 
 466944 H_{12} H_8^2 S^5
 + 196608 H_8^3 S^5 + 5898240 H_{12} H_8^2 S^6 + 
 8749056 H_8^3 S^6 + 179306496 H_8^3 S^7\\&
 + 1358954496 H_8^3 S^8 + 
 12 H_{12}^3 U + 6 H_{12}^2 H_8 U + H_{12} H_8^2 U + 288 H_{12}^3 S U + 
 576 H_{12}^2 H_8 S U + 152 H_{12} H_8^2 S U + 4 H_8^3 S U\\&
 + 
 23040 H_{12}^2 H_8 S^2 U + 11136 H_{12} H_8^2 S^2 U + 768 H_8^3 S^2 U + 
 331776 H_{12}^2 H_8 S^3 U + 516096 H_{12} H_8^2 S^3 U + 64512 H_8^3 S^3 U \\
&+ 
 13271040 H_{12} H_8^2 S^4 U 
+ 3538944 H_8^3 S^4 U 
+ 
 127401984 H_{12} H_8^2 S^5 U + 127401984 H_8^3 S^5 U + 
 2378170368 H_8^3 S^6 U\\&
 + 16307453952 H_8^3 S^7 U - 864 H_{12}^2 H_8 U^2
 - 
 288 H_{12} H_8^2 U^2 - 48 H_8^3 U^2 - 27648 H_{12} H_8^2 S U^2 - 
 5760 H_8^3 S U^2 - 663552 H_{12} H_8^2 S^2 U^2 \\&
- 331776 H_8^3 S^2 U^2
 - 
 10616832 H_8^3 S^3 U^2 
- 127401984 H_8^3 S^4 U^2 + 6912 H_{12}^2 H_8 V + 
 1152 H_{12} H_8^2 V + 4 H_{12}^2 S V + 2 H_{12} H_8 S V \\:&
+ 110592 H_{12} H_8^2 S V + 
 96 H_{12}^2 S^2 V + 336 H_{12} H_8 S^2 V
 + 32 H_8^2 S^2 V + 
 2654208 H_{12} H_8^2 S^2 V + 19968 H_{12} H_8 S^3 V + 6272 H_8^2 S^3 V \\&
+ 
 368640 H_{12} H_8 S^4 V
 + 436224 H_8^2 S^4 V + 12681216 H_8^2 S^5 V 
+ 
 127401984 H_8^2 S^6 V + 288 H_{12}^2 U V + 72 H_{12} H_8 U V + 4 H_8^2 U V\\&
 + 
 10368 H_{12}^2 S U V + 9216 H_{12} H_8 S U V + 672 H_8^2 S U V + 
 497664 H_{12} H_8 S^2 U V
 + 78336 H_8^2 S^2 U V 
+ 
 7962624 H_{12} H_8 S^3 U V \\&
+ 5308416 H_8^2 S^3 U V + 
 159252480 H_8^2 S^4 U V + 1528823808 H_8^2 S^5 U V - 
 20736 H_{12} H_8 U^2 V - 3456 H_8^2 U^2 V \\&
- 331776 H_8^2 S U^2 V - 
 7962624 H_8^2 S^2 U^2 V + H_{12} V^2 + 41472 H_{12} H_8 V^2 
+ 
 168 H_{12} S V^2 + 36 H_8 S V^2 + 5760 H_{12} S^2 V^2\\&
 + 5472 H_8 S^2 V^2 + 
 267264 H_8 S^3 V^2 + 3981312 H_8 S^4 V^2 + 2592 H_{12} U V^2 + 
 144 H_8 U V^2 + 124416 H_{12} S U V^2 + 41472 H_8 S U V^2\\&
 + 
 2985984 H_8 S^2 U V^2 + 47775744 H_8 S^3 U V^2 - 124416 H_8 U^2 V^2 + 
 12 V^3 + 1440 S V^3 + 41472 S^2 V^3 + 10368 U V^3 + 497664 S U V^3,\\
K_{42}&=-48 H_{12}^3 H_8 + 8 H_{12}^2 H_8^2 + 192 H_{12} H_8^3 S - 2 H_{12}^3 S^2 - 
 H_{12}^2 H_8 S^2 - 18432 H_{12}^2 H_8^2 S^2 + 18432 H_{12} H_8^3 S^2 - 
 64 H_{12}^3 S^3 - 112 H_{12}^2 H_8 S^3 \\&
- 16 H_{12} H_8^2 S^3 + 
 294912 H_{12} H_8^3 S^3 - 4608 H_{12}^2 H_8 S^4 - 2560 H_{12} H_8^2 S^4 - 
 128 H_8^3 S^4 - 73728 H_{12}^2 H_8 S^5 - 141312 H_{12} H_8^2 S^5 \\&
- 
 24576 H_8^3 S^5 
- 3244032 H_{12} H_8^2 S^6 - 1671168 H_8^3 S^6 - 
 28311552 H_{12} H_8^2 S^7 
- 49545216 H_8^3 S^7 - 679477248 H_8^3 S^8 - 
 3623878656 H_8^3 S^9 \\&
+ 2 H_{12}^3 U 
+ H_{12}^2 H_8 U - 2304 H_{12} H_8^3 U + 
 72 H_{12}^3 S U + 108 H_{12}^2 H_8 S U + 10 H_{12} H_8^2 S U - H_8^3 S U
 + 
 4032 H_{12}^2 H_8 S^2 U + 1632 H_{12} H_8^2 S^2 U\\&
 - 144 H_8^3 S^2 U 
+ 
 55296 H_{12}^2 H_8 S^3 U + 82944 H_{12} H_8^2 S^3 U - 2304 H_8^3 S^3 U 
 1548288 H_{12} H_8^2 S^4 U + 331776 H_8^3 S^4 U+ 
 10616832 H_{12} H_8^2 S^5 U\\&
+ 10616832 H_8^3 S^5 U + 
 84934656 H_8^3 S^6 U - 72 H_{12} H_8^2 U^2 + 12 H_8^3 U^2 - 
 3456 H_{12} H_8^2 S U^2
 - 82944 H_8^3 S^2 U^2 - 1327104 H_8^3 S^3 U^2\\&
+ 
 6912 H_8^3 U^3 - 576 H_{12}^2 H_8 V + 192 H_{12} H_8^2 V 
+ 
 9216 H_{12} H_8^2 S V - 48 H_{12}^2 S^2 V
 - 24 H_{12} H_8 S^2 V - 
 2 H_8^2 S^2 V - 2304 H_{12}^2 S^3 V\\&
 - 3072 H_{12} H_8 S^3 V - 
 608 H_8^2 S^3 V - 129024 H_{12} H_8 S^4 V - 61440 H_8^2 S^4 V - 
 1769472 H_{12} H_8 S^5 V - 2654208 H_8^2 S^5 V - 49545216 H_8^2 S^6 V\\& 
- 
 339738624 H_8^2 S^7 V 
+ 36 H_{12}^2 U V + 12 H_{12} H_8 U V - H_8^2 U V + 
 1728 H_{12}^2 S U V + 1440 H_{12} H_8 S U V - 48 H_8^2 S U V+  55296 H_{12} H_8 S^2 U V\\&
 + 6912 H_8^2 S^2 U V + 663552 H_{12} H_8 S^3 U V 
+ 
 442368 H_8^2 S^3 U V + 5308416 H_8^2 S^4 U V - 864 H_8^2 U^2 V - 
 41472 H_8^2 S U^2 V - 6 H_{12} S V^2 \\&
- 3 H_8 S V^2 - 864 H_{12} S^2 V^2 - 
 576 H_8 S^2 V^2 - 27648 H_{12} S^3 V^2 - 39168 H_8 S^3 V^2 - 
 1105920 H_8 S^4 V^2 - 10616832 H_8 S^5 V^2\\&
 + 216 H_{12} U V^2 +  36 H_8 U V^2 + 10368 H_{12} S U V^2 + 3456 H_8 S U V^2 +  82944 H_8 S^2 U V^2 - V^3 - 144 S V^3 - 6912 S^2 V^3 - 110592 S^3 V^3,\\
H_{48}&=-162 H_{12}^3 H_8^2 - 63 H_{12}^2 H_8^3 - 4 H_{12} H_8^4 - 5172 H_{12}^2 H_8^3 S - 
 834 H_{12} H_8^4 S 
- H_{12}^4 S^2 + 3 H_{12}^3 H_8 S^2 + H_{12}^2 H_8^2 S^2\\&
 - 
 124416 H_{12}^2 H_8^3 S^2
 - 61632 H_{12} H_8^4 S^2 
+ 48 H_8^5 S^2 - 
 64 H_{12}^3 H_8 S^3 + 40 H_{12}^2 H_8^2 S^3 - 12 H_{12} H_8^3 S^3 - 3 H_8^4 S^3 
- 
 1981440 H_{12} H_8^4 S^3\\&
 + 5376 H_8^5 S^3 - 1536 H_{12}^3 H_8 S^4
 - 
 1152 H_{12}^2 H_8^2 S^4 
- 1792 H_{12} H_8^3 S^4 - 624 H_8^4 S^4 - 
 23887872 H_{12} H_8^4 S^4 + 184320 H_8^5 S^4 - 79872 H_{12}^2 H_8^2 S^5 \\&
- 
 99328 H_{12} H_8^3 S^5 - 51712 H_8^4 S^5
 + 1769472 H_8^5 S^5 - 
 884736 H_{12}^2 H_8^2 S^6 - 2441216 H_{12} H_8^3 S^6 - 2170880 H_8^4 S^6
- 
 33030144 H_{12} H_8^3 S^7\\&
 - 48758784 H_8^4 S^7 - 226492416 H_{12} H_8^3 S^8
 - 
 594542592 H_8^4 S^8
 - 4529848320 H_8^4 S^9 - 21743271936 H_8^4 S^{10} + 
 H_{12}^4 U - 3 H_{12}^3 H_8 U\\&
 - H_{12}^2 H_8^2 U - 72 H_{12} H_8^4 U - 12 H_8^5 U + 
 12 H_{12}^3 H_8 S U
 - 78 H_{12}^2 H_8^2 S U
- H_{12} H_8^3 S U + H_8^4 S U- 
 1152 H_8^5 S U + 1152 H_{12}^3 H_8 S^2 U \\&
- 2496 H_{12}^2 H_8^2 S^2 U- 
 464 H_{12} H_8^3 S^2 U + 160 H_8^4 S^2 U - 27648 H_8^5 S^2 U
- 
 4608 H_{12}^2 H_8^2 S^3 U - 38912 H_{12} H_8^3 S^3 U + 7552 H_8^4 S^3 U \\&
+ 
 442368 H_{12}^2 H_8^2 S^4 U - 1290240 H_{12} H_8^3 S^4 U - 
 24576 H_8^4 S^4 U - 8847360 H_{12} H_8^3 S^5 U
 - 11501568 H_8^4 S^5 U 
+ 
 56623104 H_{12} H_8^3 S^6 U\\&
- 297271296 H_8^4 S^6 U - 
 2038431744 H_8^4 S^7 U + 288 H_{12}^2 H_8^2 U^2 + 124 H_{12} H_8^3 U^2+ 
 7 H_8^4 U^2
 + 9216 H_{12} H_8^3 S U^2 + 1488 H_8^4 S U^2\\&
 + 
 221184 H_{12} H_8^3 S^2 U^2 + 110592 H_8^4 S^2 U^2+ 
 3538944 H_8^4 S^3 U^2 + 42467328 H_8^4 S^4 U^2 - 3888 H_{12}^2 H_8^2 V
 - 
 642 H_{12} H_8^3 V\\&
 - 61920 H_{12} H_8^3 S V+ 72 H_8^4 S V - 48 H_{12}^3 S^2 V + 
 12 H_{12}^2 H_8 S^2 V - 24 H_{12} H_8^2 S^2 V - 7 H_8^3 S^2 V - 
 1492992 H_{12} H_8^3 S^2 V \\&
+ 6912 H_8^4 S^2 V - 2688 H_{12}^2 H_8 S^3 V- 
 3072 H_{12} H_8^2 S^3 V
 - 1360 H_8^3 S^3 V + 110592 H_8^4 S^3 V - 
 55296 H_{12}^2 H_8 S^4 V - 116736 H_{12} H_8^2 S^4 V \\&
- 96768 H_8^3 S^4 V - 
 2211840 H_{12} H_8^2 S^5 V - 3133440 H_8^3 S^5 V 
- 
 21233664 H_{12} H_8^2 S^6 V - 48955392 H_8^3 S^6 V\\&
 - 
 452984832 H_8^3 S^7 V
 - 2717908992 H_8^3 S^8 V + 36 H_{12}^3 U V 
- 
 48 H_{12}^2 H_8 U V - 3 H_{12} H_8^2 U V + H_8^3 U V - 864 H_8^4 U V - 
 144 H_{12}^2 H_8 S U V \\&
- 480 H_{12} H_8^2 S U V + 180 H_8^3 S U V + 
 27648 H_{12}^2 H_8 S^2 U V - 48384 H_{12} H_8^2 S^2 U V + 
 5184 H_8^3 S^2 U V - 552960 H_{12} H_8^2 S^3 U V \\&
- 387072 H_8^3 S^3 U V 
+ 
 5308416 H_{12} H_8^2 S^4 U V
 - 19906560 H_8^3 S^4 U V - 
 191102976 H_8^3 S^5 U V + 6912 H_{12} H_8^2 U^2 V
 + 1152 H_8^3 U^2 V\\&
 + 
 110592 H_8^3 S U^2 V + 2654208 H_8^3 S^2 U^2 V 
- 23328 H_{12} H_8^2 V^2 + 
 36 H_8^3 V^2 - 6 H_{12}^2 S V^2 - 15 H_{12} H_8 S V^2 - 6 H_8^2 S V^2 + 
 1728 H_8^3 S V^2\\&
 - 864 H_{12}^2 S^2 V^2 - 1296 H_{12} H_8 S^2 V^2 - 
 1032 H_8^2 S^2 V^2 
- 41472 H_{12} H_8 S^3 V^2 - 59136 H_8^2 S^3 V^2 - 
 663552 H_{12} H_8 S^4 V^2\\&
 - 1327104 H_8^2 S^4 V^2 - 
 15925248 H_8^2 S^5 V^2 - 127401984 H_8^2 S^6 V^2 
+ 432 H_{12}^2 U V^2
- 
 252 H_{12} H_8 U V^2 + 42 H_8^2 U V^2 \\&
- 8640 H_{12} H_8 S U V^2 - 
 864 H_8^2 S U V^2 + 165888 H_{12} H_8 S^2 U V^2 
- 373248 H_8^2 S^2 U V^2 
- 
 5971968 H_8^2 S^3 U V^2 + 41472 H_8^2 U^2 V^2 \\&
- H_{12} V^3 - 2 H_8 V^3 
- 
 144 H_{12} S V^3 
- 276 H_8 S V^3 - 6912 H_{12} S^2 V^3 - 12096 H_8 S^2 V^3
 - 
 221184 H_8 S^3 V^3 
- 2654208 H_8 S^4 V^3 \\&
+ 1728 H_{12} U V^3 - 
 1296 H_8 U V^3 - 62208 H_8 S U V^3 - 9 V^4
 - 864 S V^4 - 20736 S^2 V^4
\end{align*}

\begin{align*}
K_{52}&=-876 H_{12}^4 H_8 - 124 H_{12}^3 H_8^2 + H_{12}^2 H_8^3 - 21504 H_{12}^3 H_8^2 S + 
 384 H_{12}^2 H_8^3 S
 + 288 H_{12} H_8^4 S \\&
- 2 H_{12}^4 S^2 - 7 H_{12}^3 H_8 S^2 - 
 H_{12}^2 H_8^2 S^2 - 672768 H_{12}^3 H_8^2 S^2- 23040 H_{12}^2 H_8^3 S^2 \\&
+ 
 53760 H_{12} H_8^4 S^2 - 64 H_{12}^4 S^3 - 512 H_{12}^3 H_8 S^3- 
 240 H_{12}^2 H_8^2 S^3 
- 5750784 H_{12}^2 H_8^3 S^3\\&
 + 2 H_8^4 S^3 + 
 3588096 H_{12} H_8^4 S^3 - 6400 H_{12}^3 H_8 S^4 - 18432 H_{12}^2 H_8^2 S^4 
- 
 704 H_{12} H_8^3 S^4 \\&
- 129171456 H_{12}^2 H_8^3 S^4
 + 480 H_8^4 S^4 + 
 100270080 H_{12} H_8^4 S^4 - 98304 H_{12}^3 H_8 S^5 - 
 538624 H_{12}^2 H_8^2 S^5 \\&
- 131072 H_{12} H_8^3 S^5 + 43008 H_8^4 S^5+ 
 962592768 H_{12} H_8^4 S^5 - 6193152 H_{12}^2 H_8^2 S^6 
- 
 8241152 H_{12} H_8^3 S^6\\& 
+ 1630208 H_8^4 S^6 - 56623104 H_{12}^2 H_8^2 S^7 - 
 208404480 H_{12} H_8^3 S^7 
+ 10616832 H_8^4 S^7 \\&
- 
 2378170368 H_{12} H_8^3 S^8 - 1019215872 H_8^4 S^8- 14495514624 H_{12} H_8^3 S^9
 - 29595009024 H_8^4 S^9 \\&
- 
 318901321728 H_8^4 S^{10} - 1391569403904 H_8^4 S^{11} + 2 H_{12}^4 U + 
 7 H_{12}^3 H_8 U + H_{12}^2 H_8^2 U\\&
 - 18432 H_{12}^2 H_8^3 U 
- 3072 H_{12} H_8^4 U + 
 72 H_{12}^4 S U + 556 H_{12}^3 H_8 S U + 242 H_{12}^2 H_8^2 S U - 
 H_{12} H_8^3 S U \\&
- 2 H_8^4 S U 
- 294912 H_{12} H_8^4 S U + 
 6144 H_{12}^3 H_8 S^2 U + 18784 H_{12}^2 H_8^2 S^2 U + 272 H_{12} H_8^3 S^2 U \\&
- 
 528 H_8^4 S^2 U - 7077888 H_{12} H_8^4 S^2 U + 82944 H_{12}^3 H_8 S^3 U + 
 516096 H_{12}^2 H_8^2 S^3 U + 89344 H_{12} H_8^3 S^3 U\\& 
- 52480 H_8^4 S^3 U + 
 4202496 H_{12}^2 H_8^2 S^4 U + 6316032 H_{12} H_8^3 S^4 U - 
 2383872 H_8^4 S^4 U \\&
+ 31850496 H_{12}^2 H_8^2 S^5 U
 + 
 136249344 H_{12} H_8^3 S^5 U - 44236800 H_8^4 S^5 U + 
 934281216 H_{12} H_8^3 S^6 U\\&
 - 9437184 H_8^4 S^6 U + 
 4076863488 H_{12} H_8^3 S^7 U + 7247757312 H_8^4 S^7 U + 
 43486543872 H_8^4 S^8 U\\& 
+ 2304 H_{12}^3 H_8 U^2 
+ 216 H_{12}^2 H_8^2 U^2 +100 H_{12} H_8^3 U^2 + 20 H_8^4 U^2 + 51840 H_{12}^2 H_8^2 S U^2\\& 
+ 
 4032 H_{12} H_8^3 S U^2 + 2400 H_8^4 S U^2 + 1769472 H_{12}^2 H_8^2 S^2 U^2 - 
 165888 H_{12} H_8^3 S^2 U^2 + 4608 H_8^4 S^2 U^2 \\&
+ 
 11501568 H_{12} H_8^3 S^3 U^2 
- 8699904 H_8^4 S^3 U^2 + 
 339738624 H_{12} H_8^3 S^4 U^2 - 336199680 H_8^4 S^4 U^2\\&
 - 
 3227516928 H_8^4 S^5 U^2 + 55296 H_{12} H_8^3 U^3 + 9216 H_8^4 U^3 + 884736 H_8^4 S U^3 \\&
+ 21233664 H_8^4 S^2 U^3 - 21024 H_{12}^3 H_8 V - 
 240 H_{12}^2 H_8^2 V + 288 H_{12} H_8^3 V - 179712 H_{12}^2 H_8^2 S V \\&
+ 
 59904 H_{12} H_8^3 S V - 72 H_{12}^3 S^2 V - 176 H_{12}^2 H_8 S^2 V - 
 6 H_{12} H_8^2 S^2 V- 8073216 H_{12}^2 H_8^2 S^2 V \\&
+ 2 H_8^3 S^2 V + 
 3760128 H_{12} H_8^3 S^2 V - 3072 H_{12}^3 S^3 V - 13952 H_{12}^2 H_8 S^3 V -  2656 H_{12} H_8^2 S^3 V \\&
+ 480 H_8^3 S^3 V + 60162048 H_{12} H_8^3 S^3 V - 
 239616 H_{12}^2 H_8 S^4 V \\&
- 280576 H_{12} H_8^2 S^4 V + 36352 H_8^3 S^4 V - 
 3538944 H_{12}^2 H_8 S^5 V 
- 10887168 H_{12} H_8^2 S^5 V + 
 417792 H_8^3 S^5 V\\& 
- 166330368 H_{12} H_8^2 S^6 V
- 59768832 H_8^3 S^6 V 
-  1358954496 H_{12} H_8^2 S^7 V - 2378170368 H_8^3 S^7 V \\&
- 
 32614907904 H_8^3 S^8 V - 173946175488 H_8^3 S^9 V + 60 H_{12}^3 U V + 
 164 H_{12}^2 H_8 U V + 3 H_{12} H_8^2 U V \\&
- 2 H_8^3 U V - 
 221184 H_{12} H_8^3 U V+ 2592 H_{12}^3 S U V 
+ 13248 H_{12}^2 H_8 S U V + 
 1816 H_{12} H_8^2 S U V - 544 H_8^3 S U V \\
&+ 152064 H_{12}^2 H_8 S^2 U V + 
 202368 H_{12} H_8^2 S^2 U V 
- 52992 H_8^3 S^2 U V + 
 1990656 H_{12}^2 H_8 S^3 U V\\& 
+ 6967296 H_{12} H_8^2 S^3 U V 
- 
 1916928 H_8^3 S^3 U V + 61046784 H_{12} H_8^2 S^4 U V \\&
- 
 7077888 H_8^3 S^4 U V + 382205952 H_{12} H_8^2 S^5 U V 
+ 
 509607936 H_8^3 S^5 U V \\&
+ 4076863488 H_8^3 S^6 U V + 
 55296 H_{12}^2 H_8 U^2 V 
- 1728 H_{12} H_8^2 U^2 V \\&
+ 480 H_8^3 U^2 V + 
 359424 H_{12} H_8^2 S U^2 V - 96768 H_8^3 S U^2 V +  21233664 H_{12} H_8^2 S^2 U^2 V \\&
- 12607488 H_8^3 S^2 U^2 V - 
 201719808 H_8^3 S^3 U^2 V + 663552 H_8^3 U^3 V \\&
- 126144 H_{12}^2 H_8 V^2 + 
 19584 H_{12} H_8^2 V^2 
- 6 H_{12}^2 S V^2 - 3 H_{12} H_8 S V^2 + 
 940032 H_{12} H_8^2 S V^2 - 1440 H_{12}^2 S^2 V^2 \\&
- 1776 H_{12} H_8 S^2 V^2 + 
 56 H_8^2 S^2 V^2 - 55296 H_{12}^2 S^3 V^2 - 143616 H_{12} H_8 S^3 V^2 - 
 2944 H_8^2 S^3 V^2\\&
 - 3428352 H_{12} H_8 S^4 V^2 - 1155072 H_8^2 S^4 V^2- 
 42467328 H_{12} H_8 S^5 V^2 - 63700992 H_8^2 S^5 V^2\\&
- 
 1189085184 H_8^2 S^6 V^2 - 8153726976 H_8^2 S^7 V^2 + 
 648 H_{12}^2 U V^2 + 1020 H_{12} H_8 U V^2 - 128 H_8^2 U V^2 \\&
+ 
 31104 H_{12}^2 S U V^2 + 84672 H_{12} H_8 S U V^2 - 13632 H_8^2 S U V^2 + 
 1078272 H_{12} H_8 S^2 U V^2 \\&
- 193536 H_8^2 S^2 U V^2 +  11943936 H_{12} H_8 S^3 U V^2 + 10616832 H_8^2 S^3 U V^2 + 
 127401984 H_8^2 S^4 U V^2 \\
&+ 331776 H_{12} H_8 U^2 V^2 - 65664 H_8^2 U^2 V^2 - 3151872 H_8^2 S U^2 V^2\\& 
- H_{12} V^3 -  216 H_{12} S V^3- 48 H_8 S V^3 
- 17280 H_{12} S^2 V^3 - 9216 H_8 S^2 V^3 - 
 442368 H_{12} S^3 V^3\\&
 - 626688 H_8 S^3 V^3 - 17694720 H_8 S^4 V^3 - 
 169869312 H_8 S^5 V^3 + 2592 H_{12} U V^3 + 576 H_8 U V^3\\&
+  124416 H_{12} S U V^3 + 55296 H_8 S U V^3 + 1327104 H_8 S^2 U V^3 - 
 12 V^4 - 1728 S V^4 \\&
- 82944 S^2 V^4 - 1327104 S^3 V^4.
\end{align*}
}

\section*{Acknowledgement}
The idea of proof of $h_{15}\in M_{15}^+(\Gamma _0^{(1)}(4),\chi _{-4})$ using the twisting operator is due to
Professor S. B\"ocherer. 
This makes it possible to prove Lemma \ref{Lem2}. 
The author is supported by JSPS KAKENHI Grant Number JP18K03229.

\end{document}